\documentclass[a4paper,11pt]{article}
\usepackage[margin=1in]{geometry}
\usepackage[T1]{fontenc}
\usepackage{amsthm,amsfonts,amssymb,amsmath}
\usepackage{microtype}
\usepackage{graphicx}
\usepackage{xcolor}
\usepackage{enumitem}
\usepackage[nocompress]{cite}

\usepackage[colorlinks=true,urlcolor=blue,linkcolor=blue,citecolor=blue]{hyperref}

\newtheorem{Theorem}{Theorem}[section]
\newtheorem*{Itheorem}{Main results (informal summary)}
\newtheorem{Lemma}[Theorem]{Lemma}
\newtheorem{Proposition}[Theorem]{Proposition}	
\newtheorem{Corollary}[Theorem]{Corollary}	
\theoremstyle{definition}
\newtheorem{Remark}[Theorem]{Remark}

\numberwithin{equation}{section}

\newcommand{\C}{\mathbb{C}} 
\newcommand{\R}{\mathbb{R}} 
\newcommand{\Q}{\mathbb{Q}} 
\newcommand{\N}{\mathbb{N}} 
\newcommand{\per}{\textup{per}}
\renewcommand{\Re}{\operatorname{Re}}
\renewcommand{\Im}{\operatorname{Im}}
\newcommand{\iu}{\mathrm{i}}
\newcommand{\eu}{\mathrm{e}}
\newcommand{\eps}{\varepsilon}
\newcommand{\ub}{\mathbf{u}}
\newcommand{\vb}{\mathbf{v}}
\newcommand{\wb}{\mathbf{w}}
\newcommand{\El}{\mathcal{L}}
\newcommand{\de}{\mathrm{d}}
\newcommand{\A}{\mathcal{A}}
\newcommand{\ubu}{\underline{\mathbf{u}}}
\newcommand{\unu}{\underline{u}}

\DeclareMathOperator{\sign}{sign}

\DeclareMathOperator{\sech}{sech}

\allowdisplaybreaks[3]

\title{Dissipative bright and dark pulses under periodic forcing}
\author{Lukas Bengel$^*$ and Bj\"orn de Rijk\thanks{Department of Mathematics, Karlsruhe Institute of Technology, Englerstra\ss e 2, 76131 Karlsruhe, Germany; \texttt{lukas.bengel@kit.edu}, \texttt{bjoern.de-rijk@kit.edu}}}

\begin{document}

\maketitle

\begin{abstract} 
We analyze the existence and stability of bright and dark pulses in a damped nonlinear Schr\"odinger (NLS) equation with periodic forcing. This Lugiato--Lefever equation is a canonical model in nonlinear optics describing pattern formation in a dissipative Kerr cavity driven by a bichromatic laser pump. Bifurcating from the bright and black NLS solitons, we rigorously construct bright and dark single-pulse solutions to the Lugiato--Lefever equation. Using a recently developed toolbox for concatenating and periodically extending pulse solutions in spatially periodic systems, we then obtain periodic multipulse solutions corresponding to experimentally observed broad-bandwidth frequency combs. While the existence and stability of the bright pulses follow from standard Lyapunov--Schmidt reduction and Krein index arguments, the construction of the dark pulses is substantially more delicate because the nonzero asymptotic states of the black NLS soliton yield neutral essential spectrum of the linearization. By carefully tracking the associated small spatial Floquet exponents along the bifurcation and employing exponential dichotomies, we establish the existence of dark pulse solutions. These consist of two domain walls connected by a long plateau with critical absolute spectrum. To our knowledge, this constitutes the first rigorous bifurcation result in the Lugiato--Lefever equation from the black NLS soliton.

\end{abstract}

\section{Introduction}\label{sec:intro}

We investigate the existence and stability of bright and dark pulse solutions to the Lugiato--Lefever equation with periodic forcing,
\begin{align}\label{eq:LLE}
	\iu u_t = - d u_{\xi\xi} + (\zeta - \gamma\iu) u - |u|^2 u + \iu F(\xi - \omega t),
\end{align}
where $u \colon \R \times \R \to \C$ is a complex-valued function, $F \colon \R \to \C$ is a smooth periodic function, and the parameters $d \in \{\pm 1\}$, $\omega \in \R$, and $\zeta,\gamma > 0$ represent dispersion, phase velocity, detuning, and damping, respectively. The case $d = 1$ corresponds to the anomalous dispersion (focusing) regime, whereas $d = -1$ corresponds to the normal dispersion (defocusing) regime. 

Equation~\eqref{eq:LLE} arises as a model in nonlinear optics for frequency-comb generation in dissipative Kerr ringresonators driven by a bichromatic laser pump, see~\cite{Gasmi2023Bandwidth,Hansson2014Bichromatic} and Remark~\ref{rem:derivation}, or by more general spatially periodic laser pumps~\cite{Sun2025,herr2026frequencycombs}. Frequency combs are periodic optical signals whose frequency spectrum is composed of a series of equally spaced spectral lines. Combs with broad frequency spectrum have revolutionized precision frequency measurements, enabling applications in optical communications~\cite{Marin-Palomo2017}, broadband gas sensing~\cite{Schliesser2005}, spectroscopy~\cite{Coddington2008,Picque2019}, and optical metrology~\cite{Udem2002}. The experimental realization of frequency combs comprised of a multitude of well-separated cavity solitons has sparked significant interest~\cite{Herr2014Temporal,Brasch2016Photonic,Xue2015}. These solitons correspond to ultrashort optical pulses with broad frequency spectrum, which makes soliton-based frequency combs highly attractive for applications~\cite{Marin-Palomo2017,Weiner2017}.

\emph{Soliton-based frequency combs} in the Lugiato--Lefever model~\eqref{eq:LLE} are traveling periodic wave solutions obtained by periodically extending bright and dark pulse solutions; see Figure~\ref{fig:pulses_intro} for an illustration. \emph{Bright and dark pulses} are traveling pulse solutions to~\eqref{eq:LLE} that arise near the nonlinear Schr\"odinger (NLS) limit, obtained by setting the damping $\gamma$ and forcing $F$ equal to $0$. Their profiles approximate the formal concatenation
\begin{align} \label{eq:formal_concat}
\sum_{i = 1}^N \chi_i(x) \phi_{\theta_i}(x-X_i),
\end{align}
representing a superposition of $N \in \mathbb{N}$ solitons $\phi_{\theta_i}$ solving the NLS equation
\begin{align} \label{NLS}
\iu u_t = -d u_{xx} + \zeta u - |u|^2 u.
\end{align}
Here, $X_i \in \R$ with $X_i \neq X_j$ for $i \neq j$ denotes the position of the $i$-th soliton, and the functions $\chi_i \colon \R \to [0,1]$ form a partition of unity, with $\chi_i$ being supported on an interval containing $X_i$ and satisfying $\chi_i(X_i) = 1$ for $i = 1,\ldots,N$. In the anomalous dispersion regime,
\begin{align} \label{bright_sol}
\phi_{\theta}(x) = \sqrt{2\zeta} \sech\left(\sqrt{\zeta}\, x\right) \eu^{\iu \theta}
\end{align}
is the stationary \emph{bright-soliton solution} to~\eqref{NLS} with phase $\theta \in \R$, whereas in the normal dispersion regime
\begin{align} \label{black_sol}
\phi_{\theta}(x) = \sqrt{\zeta} \tanh\left(\sqrt{\tfrac12 \zeta} \, x\right) \eu^{\iu \theta}
\end{align}
is the stationary \emph{black-soliton (or domain-wall) solution} to~\eqref{NLS}. Owing to its nonzero asymptotic states, the phases must satisfy the matching condition $\theta_i + \pi = \theta_{i+1}$ for $i = 1,\ldots,N-1$ in the normal dispersion regime.

\begin{figure}[b]
    \centering
    \includegraphics[width=.85\textwidth]{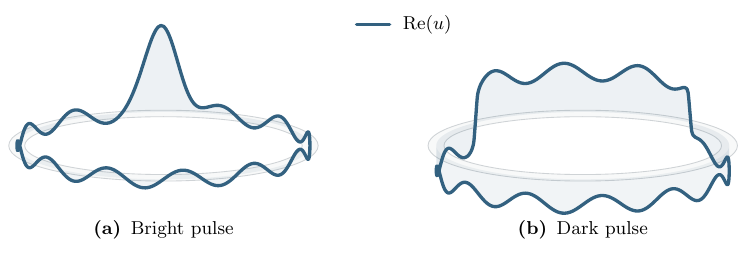}
    \caption{Illustration of soliton-based frequency-comb solution on a ringresonator, consisting of a single bright and dark pulse. The bright pulse is supported on a periodic background wave, while the dark pulse is formed by two domain walls connecting different periodic states.}
    \label{fig:pulses_intro}
\end{figure}

In contrast to standard monochromatic pumping schemes, for which the forcing function $F$ is constant, bichromatic or multi-mode pumping breaks the translational symmetry of the Lugiato--Lefever equation~\eqref{eq:LLE} and leads to pinning of pulse solutions~\cite{BengelReichel2024,Bengel2026Dynamics}, thereby stabilizing the comb repetition rate and facilitating robust (single-)soliton operation~\cite{Taheri2017Optical,Cole2018KerrChirped}.
Moreover, experimental and numerical studies indicate that multi-mode pumping produces combs with broader frequency spectrum, resulting in improved performance in a variety of applications~\cite{Gasmi2023Bandwidth}.

Despite the advantages of multi-mode pumping for the generation of soliton-based frequency combs, the authors are not aware of any mathematically rigorous studies addressing the existence and stability of bright and dark pulse solutions to~\eqref{eq:LLE} when the forcing term $F$ is genuinely periodic. The primary objective of this paper is to fill this gap by constructing traveling single-pulse solutions to~\eqref{eq:LLE} that bifurcate from the bright and black solitons of the NLS equation for general periodic forcing functions $F$. By concatenating the resulting single-pulse solutions, we obtain bright and dark multipulse solutions to~\eqref{eq:LLE}. Extending them periodically, we arrive at soliton-based frequency-comb solutions consisting of an arbitrary number of well-separated, strongly localized, and highly nonlinear pulses within a single periodicity interval, thereby generating broad frequency spectrum. In addition, we perform a spectral stability analysis of the bright pulse solutions, yielding strong nonlinear stability properties with exponential decay of perturbations.

Our main result may be informally summarized as follows.

\begin{Itheorem} Fix parameters $\zeta,T > 0$. Let $F \in C^1(\R,\C)$ be $T$-periodic.
\begin{itemize}
    \item[(i)] \emph{(Bright and dark primary pulses).} Provided an explicit existence and nondegeneracy condition in terms of $T$, $\zeta$, $F$, and $\phi_0$ is satisfied, the Lugiato--Lefever equation~\eqref{eq:LLE} admits bright or dark single-pulse solutions near the NLS limit. These traveling waves propagate with speed $\omega \in \R$, and their profiles are well-approximated by~\eqref{eq:formal_concat}: in the anomalous dispersion regime, $N=1$ and $\phi_{\theta_1}(\cdot-X_1)$ is a shifted and phase-rotated bright soliton~\eqref{bright_sol}, while in the normal dispersion regime, $N=2$ and the components $\phi_{\theta_1}(\cdot-X_1)$ and $\phi_{\theta_2}(\cdot-X_2)$ are two domain walls of the form~\eqref{black_sol}.
    \item[(ii)] \emph{(Bright and dark frequency combs).} For any formal concatenation of an arbitrary number of well-separated bright or dark primary pulse solutions, there exist soliton-based frequency-comb solutions to~\eqref{eq:LLE} with speed $\omega$ whose profiles are well-approximated by this concatenation on a single periodicity interval. Their wavelength may be chosen arbitrarily large.
    \item[(iii)] \emph{(Stability in the anomalous dispersion regime).} Provided a nondegeneracy condition is satisfied, the stability of the bright primary pulses is determined by the eigenvalues of a $(2\times2)$-matrix. Both the nondegeneracy condition and the matrix are given explicitly in terms of $\zeta$, $F$, and $\phi_0$. Moreover, the soliton-based frequency comb is stable if and only if all of its constituent bright primary pulses are stable.
\end{itemize}
\end{Itheorem}

In Section~\ref{sec:main_results}, we present the precise statements of our main results and discuss their implications, their embedding in the mathematical literature, and the proof strategy. In particular, we will see that our results imply that~\eqref{eq:LLE} admits soliton-based frequency combs comprised of ultrashort pulses. Moreover, they provide the first rigorous construction of dark pulses in the Lugiato--Lefever equation bifurcating from the black NLS soliton.

\begin{Remark} \label{rem:derivation}
In its simplest form, the dynamics in a dissipative Kerr ringresonator driven by a bichromatic laser pump is modeled by the equation~\cite{Gasmi2023Bandwidth}
\begin{align*} 
	\iu z_t = -d z_{\xi\xi} - \iu \gamma z - |z|^2 z + \iu \left(f_0\eu^{\iu(k_0 \xi - \omega_0 t) } + f_1\eu^{\iu(k_1 \xi - \omega_1 t) }\right),
\end{align*}
where $z \colon \R \times \R \to \C$ denotes the optical field, $\gamma>0$ is the damping coefficient, $d \in \{\pm 1\}$ is the dispersion coefficient, and $f_{0,1}, k_{0,1},\omega_{0,1}\in\R$ with $k_0 \neq k_1$ are the amplitudes, spatial wavenumbers, and temporal frequencies of the two pump components, respectively, satisfying $dk_0^2-\omega_0>0$. Exploiting the Galilean invariance of the Schrödinger operator together with the gauge invariance of the Kerr nonlinearity, we find that
\begin{align*}
u(\xi,t)=\eu^{-\iu(k_0 \xi +(2dk_0^2-\omega_0)t)}z(\xi+2dk_0t,t)
\end{align*}
satisfies the Lugiato--Lefever equation~\eqref{eq:LLE} with detuning $\zeta=dk_0^2-\omega_0$, phase velocity $\omega=(\omega_1-\omega_0)/(k_1-k_0)$, and periodic forcing function
\begin{align*}
F(x)=f_0+f_1\eu^{\iu(k_1-k_0)x}.
\end{align*}
Thus, the dynamics of a dissipative Kerr ringresonator driven by a bichromatic laser pump is governed by the Lugiato--Lefever equation~\eqref{eq:LLE} with genuinely periodic forcing function $F$.
\end{Remark}

\begin{Remark}
When the periodic forcing function $F$ is given by a snoidal, cnoidal, or dnoidal Jacobi elliptic function, explicit stationary periodic solutions to~\eqref{eq:LLE} that are complex scalar multiples of $F$ were obtained in~\cite{Sun2025}. The dnoidal-wave solutions, which arise in the anomalous dispersion regime, converge to strongly localized bright NLS solitons in the long-wavelength limit. In contrast, the profiles of the snoidal-wave solutions, which arise in the normal dispersion regime, converge in the long-wavelength limit to two domain walls connected by a spatially extended plateau. Consequently, these periodic solutions may be viewed as soliton-based frequency combs with broad frequency spectrum. Numerical investigations in~\cite{Sun2025} indicate that the stability of the snoidal, cnoidal, and dnoidal waves critically depends on the system parameters. The rigorous stability results established in the present paper confirm this behavior for soliton-based frequency combs in the anomalous dispersion regime.
\end{Remark}

\begin{Remark}
In the case of a constant forcing function $F$, soliton-based frequency-comb solutions to~\eqref{eq:LLE} have been rigorously constructed in the anomalous dispersion regime by bifurcating either from periodic dnoidal-wave solutions~\cite{Hakkaev2019Generation} or from bright soliton solutions~\cite{Bengel2025Existence} of the NLS equation. These frequency-comb solutions are orbitally stable with respect to co-periodic perturbations~\cite{Stanislavova2018Asymptotic}, which induce a nontrivial asymptotic phase shift due to the translational invariance of the equation. Moreover, the solutions bifurcating from bright solitons are stable against perturbations of arbitrary wavelength~\cite{Bengel2025Existence,Haragus2024Nonlinear}, whereas those bifurcating from dnoidal waves are unstable with respect to long-wavelength perturbations~\cite{Bengel2025Existence}. To the best of the authors' knowledge, the rigorous construction of soliton-based frequency combs for~\eqref{eq:LLE} in the normal dispersion regime remains an open problem in the case of constant forcing.
\end{Remark}

\paragraph*{Function spaces.} Let $T > 0$. We denote by $H^k(\R)$ and $H_\per^k(0,T)$, with $k \in \mathbb{N}_0$, the Sobolev spaces of complex-valued functions over the field $\mathbb{C}$. The $L^2$-inner product is given by  
$$
\langle u,v \rangle_{L^2} = \int_\R u(x) \overline{v(x)} \de x.
$$  
In~\S\ref{sec:existence_bright}, we consider the function spaces $H_\R^k(\mathbb{R})$ and $H_{\R,\per}^k(0,T)$ with $k \in \mathbb{N}_0$ over the field $\R$. Here, the inner product on the vector space $L_\R^2(\R)$ of complex-valued functions over $\R$ is defined by  
$$
\langle u,v \rangle_{L_\R^2} = \Re \int_\R u(x) \overline{v(x)} \de x.
$$
Note that as sets, the equalities $H^k(\R) = \smash{H_\R^k(\mathbb{R})}$ and $\smash{H_{\per}^k(0,T)}=\smash{H_{\R,\per}^k(0,T)}$ hold, with coinciding norms. The main reason for us to introduce the spaces $\smash{H_\R^k(\mathbb{R})}$ and $\smash{H_{\R,\per}^k(0,T)}$ is, that for $k \in \N$ the function $H^k(\R) \to L^2(\R),\ u \mapsto |u|^2 u$ is not Fr\'echet differentiable whereas $H_\R^k(\R) \to L_\R^2(\R),\ u \mapsto |u|^2 u$ is.

\paragraph*{Outline of paper.} This paper is organized as follows. In Section~\ref{sec:main_results}, we precisely formulate our main results. Section~\ref{sec:existence_bright} is devoted to the construction of bright primary pulse solutions. Their spectral analysis, which yields the proofs of Theorem~\ref{thm:main_bright} and Corollary~\ref{cor:bright_FC}, is presented in Section~\ref{sec:stability}. Section~\ref{sec:existence_black} is devoted to the construction of dark pulse solutions and contains the proofs of Theorem~\ref{thm:main_dark} and Corollary~\ref{cor:dark_FC}. In Section~\ref{sec:numerics}, we present numerical simulations that corroborate our existence and stability results. Finally, some technical computations are delegated to Appendix~\ref{app:computations}.

\paragraph*{Acknowledgments.} We would like to thank the organizers of the workshop \emph{Dynamics of Coherent Structures in Discrete and Continuum Nonlinear Systems} at the Institute of Mathematics at the University of Granada, held June 8--13, 2025, which provided important motivation for this work. This project is funded by the Deutsche Forschungsgemeinschaft (DFG, German Research Foundation) -- Project-ID 258734477 -- SFB 1173. 

\section{Main results} \label{sec:main_results}

In this section, we precisely formulate our main results, which establish the existence and stability of traveling bright and dark pulse solutions to the Lugiato--Lefever equation~\eqref{eq:LLE} for general forcing functions $F$. By periodically extending these pulse solutions, we then obtain soliton-based frequency combs. Since the forcing function propagates with phase velocity $\omega$, it is natural to seek bright and dark pulse solutions traveling at the same speed. Passing to the co-moving frame $x=\xi-\omega t$, we are therefore led to look for stationary solutions to
\begin{align}\label{eq:LLE2}
\iu u_t = - d u_{xx} + \iu\omega u_x + (\zeta - \gamma\iu) u - |u|^2 u + \iu F(x).
\end{align}

\subsection{Anomalous dispersion regime}

We begin by stating our main result for the anomalous dispersion regime, where we establish the existence and spectral stability of pulse solutions bifurcating from the bright NLS soliton~\eqref{bright_sol}. To this end, we set $d=1$, rescale the phase velocity $\omega$, the damping $\gamma$, and the forcing function $F$ in~\eqref{eq:LLE2} by a small parameter $\eps>0$, and assume without loss of generality that $\gamma=1$, yielding
\begin{align}\label{eq:LLE_main}
\iu u_t =- u_{xx} + \zeta u - |u|^2 u + \iu \eps\left(\omega u_x - u + F(x) \right).
\end{align}
For $\eps=0$, equation~\eqref{eq:LLE_main} then reduces to the focusing NLS equation.

To study the spectral stability of bright pulse solutions to~\eqref{eq:LLE_main}, we write it as as a system 
\begin{align}\label{eq:LLE_main_systems}
\ub_t = J\left(-\ub_{xx} + \zeta \ub - |\ub|^2 \ub\right) + \eps \left(\omega \ub_x - \ub + \mathbf{F}(x)\right),
\end{align}
in $\ub = (\Re(u), \Im(u))^\top$, where we denote
\begin{align} \label{eq:defJF}
	J = 
	\begin{pmatrix}
		0 & 1 \\ -1 & 0
	\end{pmatrix}, \qquad \mathbf{F}(x) = 
    \begin{pmatrix} \Re(F(x)) \\ \Im(F(x))\end{pmatrix}
\end{align}
and $|\ub| = \sqrt{\ub_1^2 + \ub_2^2}$ is the usual Euclidean norm. The linearization of~\eqref{eq:LLE_main_systems} about a smooth bounded stationary solution $\ubu \colon \R \to \R^2$ is then given by $\El_\eps(\ubu) - \eps$, where
\begin{align}\label{def:linearization_localized}
	L_\eps(\ubu) = -\partial_x^2 + \zeta -
	\begin{pmatrix}
	   3 \ubu_1^2 +\ubu_2^2 & 2 \ubu_1 \ubu_2 \\
	   2 \ubu_1 \ubu_2 &  \ubu_1^2 + 3 \ubu_2^2
	\end{pmatrix} - \eps \omega J \partial_x, \qquad \El_\eps(\ubu) = JL_\eps(\ubu).
\end{align}
As a composition of a skew-symmetric operator $J = - J^*$ and a self-adjoint operator $L_\eps(\ubu) = L_\eps(\ubu)^*$, the spectrum of the Hamiltonian linear operator $\El_\eps(\ubu)$ on $L^2(\R)$ possesses a symmetry with respect to the real and imaginary axis, cf.~\cite[Proposition 5.1.2]{KapitulaPromislow2013}.

Our main theorem establishes the existence and spectral stability of stationary single-pulse solutions to~\eqref{eq:LLE_main}. The proof is based on Lyapunov--Schmidt reduction, which reduces both questions to the analysis of the \emph{Melnikov integral}
\begin{align*}
\mathcal{M}_{\theta,\sigma} = \Re \langle\eu^{-\iu\theta} F(\cdot+\sigma),\phi_0\rangle_{L^2} + \iu 
\Im \langle\eu^{-\iu\theta} F'(\cdot+\sigma),\phi_0\rangle_{L^2}
\end{align*}
and the \emph{reduced Evans function} $D_{\theta,\sigma} \colon \C \to \C$ given by 
\begin{align} \label{eq:red_Evans}
    D_{\theta,\sigma}(\lambda) = \det(A_{\theta,\sigma}+\lambda^2B)
\end{align}
for $\theta,\sigma \in \R$, where we denote
\begin{align} \label{eq:explicit_matrices}
    A_{\theta,\sigma} =
    \begin{pmatrix}
        -\!\Im \langle\eu^{-\iu\theta}F''(\cdot+\sigma),\phi_0\rangle_{L^2} &
        -\!\Re \langle\eu^{-\iu\theta} F'(\cdot+\sigma),\phi_0\rangle_{L^2}\\
        -\!\Re \langle\eu^{-\iu\theta} F'(\cdot+\sigma),\phi_0\rangle_{L^2}&
         \Im\langle \eu^{-\iu\theta} F(\cdot+\sigma),\phi_0\rangle_{L^2}
    \end{pmatrix}, \qquad
    B =
    \begin{pmatrix}
        \sqrt{\zeta} & 0 \\ 0 & -\frac{1}{\sqrt{\zeta}}
    \end{pmatrix},
\end{align}
and $\phi_0$ is the bright soliton given by~\eqref{bright_sol}.

\begin{Theorem}[Bright primary pulse]\label{thm:main_bright}
Fix parameters $\zeta,T > 0$ and $\omega \in \R$. Let $F \in C^2(\R,\C)$ be $T$-periodic. Assume that $(\theta_0,\sigma_0) \in \R^2$ obeys the existence condition $\mathcal{M}_{\theta_0,\sigma_0} = 4\sqrt{\zeta}$ and the nondegeneracy condition $D_{\theta_0,\sigma_0}(0) \neq 0$. Then, there exist constants $C,\eps_0> 0$ such that for all $\eps \in (0,\eps_0)$ there exists a stationary solution $\unu_\eps \in H^2(\R) \oplus H_\per^2(0,T)$ to~\eqref{eq:LLE_main} with the following properties:
\begin{itemize}
    \item[(i)] \emph{(Approximation).} There exists a stationary background solution $v_\eps \in H_\per^2(0,T)$ to~\eqref{eq:LLE_main} such that
\begin{align} \label{eq:approximation_bright}
	\|\unu_\eps - \phi_{\theta_0}(\cdot + \sigma_0) - v_\eps\|_{H^2} \leq C \eps
\end{align}
and
\begin{align}\label{eq:approximation_background_wave}
\left\|v_\eps + \iu \eps (-\partial_x^2+\zeta)^{-1} F \right\|_{H_\per^2(0,T)} \leq C \eps^2. 
\end{align}
\item[(ii)] \emph{(Spectral stability).} If $D_{\theta_0,\sigma_0}$ has four distinct purely imaginary roots, then there exists $\mu \in (0,\eps)$ such that the spectrum of the linearization $\El_\eps(\ubu_\eps) - \eps \colon H^2(\R) \to L^2(\R)$ of~\eqref{eq:LLE_main_systems} about $\ubu_\eps = (\Re(\unu_\eps),\Im(\unu_\eps))^\top$ is confined to $\{\lambda \in \C : \Re(\lambda) \leq -\mu\}$.
\item[(iii)] \emph{(Spectral instability).} If $D_{\theta_0,\sigma_0}$ has a root in $\C \setminus \iu\R$, then there exists $\lambda \in \sigma(\El_\eps(\ubu_\eps) - \eps)$ with $\Re(\lambda) > 0$.
\end{itemize}
\end{Theorem}

\begin{Remark} \label{rem:nonl_stab}
We emphasize that the spectral stability assertion in~(ii) of Theorem~\ref{thm:main_bright} translates into linear stability by combining the resolvent estimates established in~\cite{Bengel2024Stability} with the Pr\"uss theorem~\cite{Pruess1984}. Specifically, following the linear stability analysis in~\cite[Section~3.2]{Bengel2024Stability} verbatim, one obtains constants $M_0,\mu_1 > 0$ such that the linearization $\El_\eps(\ubu_\eps) - \eps$ generates a $C_0$-semigroup on $H^1(\R)$ obeying the exponential bound
\begin{align*}
\big\|\eu^{(\El_\eps(\ubu_\eps) - \eps)t}\big\|_{H^1 \to H^1} \leq M_0\eu^{-\mu_1 t}
\end{align*}
for $t \geq 0$. Consequently, classical arguments, cf.~\cite[Theorem~10.2.2]{Cazenave1998Anintroduction}, yield asymptotic nonlinear stability with exponential decay of perturbations. In particular, there exist constants $C_0,\delta,\mu_2 > 0$ such that, for every $v_0 \in H^2(\R)$ satisfying $\|v_0\|_{H^1} \leq \delta$, there exists a global classical solution
\begin{align*}
u \in C\big([0,\infty),H^2(\R) \oplus H^2_\per(0,T)\big) \cap C^1\big([0,\infty),L^2(\R) \oplus L^2_\per(0,T)\big)
\end{align*}
to~\eqref{eq:LLE_main} with $u(0) = \unu_\eps + v_0$ such that
\begin{align*}
\|u(t) - \unu_\eps\|_{H^1} \leq C_0\eu^{-\mu_2 t} \|v_0\|_{H^1}
\end{align*}
for all $t \geq 0$.
\end{Remark}

Using the recently developed theory~\cite{Bengel2026Multiple} for concatenating and periodically extending pulse solutions in spatially periodic systems, we obtain large-wavelength soliton-based frequency-comb solutions to~\eqref{eq:LLE_main} consisting of an arbitrary number of bright primary pulse solutions per period.

\begin{Corollary}[Bright frequency combs] \label{cor:bright_FC}
Let $M \in \N$. Fix parameters $\zeta,T > 0$, and $\omega \in \R$. Let $F \in C^1(\R,\C)$ be $T$-periodic. Assume that $(\theta_i,\sigma_i) \in \R^2$ obeys the existence condition $\mathcal{M}_{\theta_i,\sigma_i} = 4\sqrt{\zeta}$ and the nondegeneracy condition $D_{\theta_i,\sigma_i}(0) \neq 0$ for all $i \in \{1,\ldots,M\}$. 

Then, there exist constants $C,\eps_0 > 0$ such that for all $\eps \in (0,\eps_0)$ there exist constants $K_\eps,N_\eps \in \N$, satisfying $K_\eps,N_\eps \to \infty$ as $\eps \downarrow 0$, such that for all $n \in \N$ there exists a stationary solution $\unu_{\eps,n} \in H^2_\per(0,\ell_{\eps,n})$ of period $\ell_{\eps,n} := (N_{\eps} + n)T$ to~\eqref{eq:LLE_main} with the following properties:
\begin{itemize}
\item[(i)] \emph{(Approximation).} There exists a stationary background solution $v_\eps \in H_\per^2(0,T)$ to~\eqref{eq:LLE_main} satisfying~\eqref{eq:approximation_background_wave} such that
\begin{align*}
\sup_{x \in [0,\ell_{\eps,n}]}	\left|\unu_{\eps,n}(x) - \sum_{i = 1}^M \phi_{\theta_i}(x + \sigma_i - i K_\eps T) - v_\eps(x)\right| \leq C \eps.
\end{align*}
\item[(ii)] \emph{(Spectral stability).} If $D_{\theta_i,\sigma_i}$ has four distinct purely imaginary roots for all $i \in \{1,\ldots,M\}$, then there exists $\mu > 0$ such that for each $k \in \N$ the spectrum of the linearization $\El_\eps(\ubu_{\eps,n}) - \eps \colon H^2_\per(0,k\ell_{\eps,n}) \to L^2_\per(0,k\ell_{\eps,n})$ of~\eqref{eq:LLE_main_systems} about $\ubu_{\eps,n} = (\Re(\unu_{\eps,n}),\Im(\unu_{\eps,n}))^\top$ is confined to $\{\lambda \in \C : \Re(\lambda) \leq -\mu\}$.
\item[(iii)] \emph{(Spectral instability).} If $D_{\theta_i,\sigma_i}$ has a root in $\C \setminus \iu\R$ for some $i \in \{1,\ldots,M\}$, then there exists $\lambda$ in the spectrum of $\El_\eps(\ubu_{\eps,n}) - \eps \colon H^2_\per(0,\ell_{\eps,n}) \to L^2_\per(0,\ell_{\eps,n})$ with $\Re(\lambda) > 0$. 
\end{itemize}
\end{Corollary}

\begin{Remark}
The analysis in~\cite{Bengel2026Multiple} shows that, for fixed $\eps>0$, the $M-1$ distances between the constituent primary pulses may be chosen arbitrarily large independently of each other. Together with the parameter $\eps$, the wavelength, and the invariance under the translation $x \mapsto x + T$, this yields an $(M+2)$-parameter family of soliton-based frequency-comb solutions to~\eqref{eq:LLE_main}. For simplicity of exposition, however, we restrict ourselves to the two-parameter family $\unu_{\eps,n}$ in Corollary~\ref{cor:bright_FC}.
\end{Remark}

After rescaling their periods to $1$, the frequency combs in Corollary~\ref{cor:bright_FC} consist of ultrashort bright pulses. Specifically, by choosing $n$ sufficiently large, the individual pulses comprising the frequency comb can be made arbitrarily localized without changing their amplitude.

Theorem~\ref{thm:main_bright} and Corollary~\ref{cor:bright_FC} show that, provided the reduced Evans functions associated with the primary pulses have no higher-multiplicity purely imaginary roots, the frequency-comb solutions inherit the spectral stability properties of their constituent primary pulses. In particular, if all reduced Evans functions have four distinct purely imaginary roots, then the resulting frequency combs are strongly spectrally stable, with a spectral gap of size $\mu > 0$. With the same arguments as in Remark~\ref{rem:nonl_stab}, this spectral gap implies asymptotic nonlinear stability with respect to subharmonic perturbations, with exponential convergence to the underlying frequency comb. Thus, in agreement with experimental observations and numerical simulations~\cite{Taheri2017Optical,Bengel2026Dynamics}, the frequency combs are pinned in the spatial domain. This contrasts sharply with soliton-based frequency combs of the Lugiato--Lefever equation with constant forcing, for which perturbations may induce a nontrivial asymptotic phase shift~\cite{Bengel2025Existence}.

\subsection{Normal dispersion regime}

We present our main result for the normal dispersion regime $d=-1$, where we establish the existence of dark pulse solutions bifurcating from the formal concatenation of domain walls~\eqref{black_sol}. The bifurcation argument proceeds in two stages: we first switch on the forcing and subsequently the damping. To facilitate this construction, we scale the forcing function $F$ in~\eqref{eq:LLE2} by a different small parameter than the phase velocity $\omega$ and the damping $\gamma$. Thus, after substituting $u \mapsto \iu u$, introducing small parameters $\eps,\eta>0$ and, without loss of generality, normalizing $\gamma=1$, we arrive at
\begin{align}\label{eq:LLE_main_tw_mod}
\iu u_t = u_{xx} + \zeta u - |u|^2 u + \iu \eps \left(\omega u_x -u\right) + \eta F(x).
\end{align}
For $\eps=\eta=0$, equation~\eqref{eq:LLE_main_tw_mod} then reduces to the defocusing NLS equation.

As in the anomalous dispersion regime, we first establish the existence of a single pulse solution, which then serves as a building block for the construction of large-wavelength soliton-based frequency combs. Our main theorem reduces the existence of a dark primary pulse solution to the study of the \emph{Melnikov integral} $\mathcal{M} \colon \R \to \R$ given by
\begin{align} \label{eq:def_Melnikov}
\mathcal{M}(\sigma) = \langle F(\cdot+\sigma),\psi'\rangle_{L^2},
\end{align}
where
\begin{align} \label{black_sol2}
\psi(x) = \sqrt{\zeta} \tanh\left(\sqrt{\tfrac12 \zeta} \, x\right)
\end{align}
is the black-soliton solution to the NLS equation~\eqref{NLS} connecting $-\sqrt{\zeta}$ to $\sqrt{\zeta}$. 

\begin{Theorem}[Dark primary pulse] \label{thm:main_dark} Fix parameters $\zeta,T > 0$ and $\omega \in \R$. Let $F \in C^1(\R,\R)$ be $T$-periodic with
\begin{align} \label{eq:F_mean}
\int_0^T F(x) \de x < 0.
\end{align}
Assume that $\sigma_0 \in \R$ obeys the existence and nondegeneracy condition
\begin{align} \label{eq:exist_dark}
\mathcal{M}(\sigma_0) = 0, \qquad \mathcal{M}'(\sigma_0) \neq 0.
\end{align}
Then, there exist a constant $C > 0$ and sequences $\{\eta_j\}_{j \in \N},\{X_j\}_{j \in \N}, \{\eps_j\}_{j \in \N} \subset (0,\infty)$ satisfying
\begin{align*}
\lim_{j \to \infty} \eta_j = 0, \qquad \lim_{j \to \infty} X_j = \infty, \qquad \lim_{j \to \infty} \eps_j = 0,
\end{align*}
such that for all $j \in \N$ there exists a stationary solution $\unu_j \in H^2(\R) \oplus H^2_\per(0,T)$ to~\eqref{eq:LLE_main_tw_mod} with $(\eta,\eps) = (\eta_j,\eps_j)$ enjoying the estimate
\begin{align*}
\left|\unu_j(\pm x) + \psi(x - X_j)\right| \leq \frac{C}{j}, 
\end{align*}
for $x \geq 0$.
\end{Theorem}

\begin{Remark}
We note that~\eqref{eq:F_mean} is, in fact, a nondegeneracy condition requiring that $F$ has nonzero mean. Indeed, if $F$ has positive mean, the coordinate transformation $u \mapsto -u$ in~\eqref{eq:LLE_main_tw_mod} flips the sign of the forcing function and~\eqref{eq:F_mean} holds.
\end{Remark}

\begin{Remark}
Theorem~\ref{thm:main_dark} establishes a one-parameter family of dark pulse solutions. Our analysis shows that, for fixed $\eta = \eta_j$, the distance $2K$ between the domain walls can be chosen arbitrarily large, while the parameter $\eps>0$ can be chosen arbitrarily small independently, thereby yielding a three-parameter family of solutions. For simplicity of presentation, we confine ourselves to the one-parameter family $\unu_j$ in Theorem~\ref{thm:main_dark}.
\end{Remark}

The dark pulse solution established in Theorem~\ref{thm:main_dark} consists of two domain walls: a front interface connecting the periodic end state near $-\sqrt{\zeta}$ to a long periodic plateau state near $\sqrt{\zeta}$, and a second interface connecting the plateau state back to the end state. Consequently, the solution represents a single dark pulse superposed on a small-amplitude periodic background.

To the best of the authors' knowledge, Theorem~\ref{thm:main_dark} provides the first rigorous existence result for dark pulse solutions to the Lugiato--Lefever equation bifurcating from the black NLS soliton. Previous bifurcation results from the black NLS soliton concern conservative perturbations~\cite{Pelinovsky2008Dark} 
or dissipative perturbations~\cite{Lega1997,Lega2001,Kapitula2000} which preserve the gauge invariance of the NLS equation, allowing to naturally identify the phase. A key challenge in proving Theorem~\ref{thm:main_dark} is that the damping and forcing in~\eqref{eq:LLE} break both the variational structure and the gauge symmetry of the NLS equation.

As in the case of bright pulse solutions in the anomalous dispersion regime, we employ the framework developed in~\cite{Bengel2026Multiple} to concatenate and periodically extend dark primary pulses, thereby obtaining soliton-based frequency-comb solutions. This approach yields frequency combs consisting of an arbitrary number of well-separated dark pulses within a single periodicity interval, and hence with arbitrarily broad frequency spectrum.

\begin{Corollary}[Dark frequency combs] \label{cor:dark_FC}
Let $M \in \N$. Fix parameters $\zeta,T > 0$ and $\omega \in \R$. Let $F \in C^1(\R,\R)$ be $T$-periodic and satisfy~\eqref{eq:F_mean}. Assume that $\sigma_0 \in \R$ obeys the existence and nondegeneracy condition~\eqref{eq:exist_dark}. 

Then, there exist a constant $C > 0$ such that for all $j \in \N$ there exist constants $K_j,N_j \in \N$ satisfying $K_j,N_j \to \infty$ as $j \to \infty$, such that for all $n \in \N$ there exists a stationary solution $\unu_{j,n} \in H^2_\per(0,\ell_{j,n})$ of period $\ell_{j,n} := (N_{j}+n)T$ to~\eqref{eq:LLE_main_tw_mod} enjoying the estimate
\begin{align*}
\sup_{x \in [0,\ell_{j,n}]} \left|\unu_{j,n}(x) - \sum_{i = 1}^M \unu_j(x - i K_j T)\right| \leq \frac{C}{j},
\end{align*}
where $\unu_j$ is the dark primary pulse solution, established in Theorem~\ref{thm:main_dark}.
\end{Corollary}

\begin{Remark}
Theorem~\ref{thm:main_dark} does not address the stability of the dark pulse solutions, which we leave for future work. We expect the stability analysis to be particularly delicate, primarily because the origin belongs to the absolute spectrum associated with the long periodic plateau state connecting the front and back interfaces of the dark primary pulse; see~\S\ref{sec:existence_black} for more details. Previous works~\cite{Sandstede2000,Sandstede2000Gluing,Goh2022,Carter2021} considering traveling-pulse solutions which such long plateau states show that, as the plateau length tends to infinity, eigenvalues of the linearization about the pulse accumulate onto the absolute spectrum associated with the plateau state. We expect that tracking these eigenvalues near zero in the stability analysis of the dark pulse solutions could be particularly challenging, cf.~\cite{Goh2022}.
\end{Remark}

\subsection{Method of proof and challenges} \label{sec:method_of_proof}

The proof of Theorem~\ref{thm:main_bright} follows along the lines of the bifurcation analysis in~\cite{Bengel2024Stability} for the case of a constant forcing function $F$. The existence argument exploits the fact that the linearization of the NLS equation about its bright soliton solution possesses an isolated eigenvalue at $0$ arising from translational and rotational invariance. A Lyapunov--Schmidt reduction argument then yields a two-dimensional bifurcation equation $\mathcal{M}_{\theta,\sigma} = 4\sqrt{\zeta}$ determining the translation and phase shift of the bright soliton~\eqref{bright_sol} at which pulse solutions bifurcate. 

To analyze the spectral stability of the bifurcating  pulse solutions, we employ a Krein index counting argument~\cite{KapitulaKevrekidis2004,AddendumKapitulaKevrekidis2004}, which shows that any unstable spectrum of the linearization of~\eqref{eq:LLE_main_systems} about the pulse must originate from four critical eigenvalues bifurcating from the the origin. A standard, but tedious, Lyapunov--Schmidt reduction procedure then leads to the reduced Evans function $D_{\theta,\sigma}$ that determines the leading-order location of these eigenvalues, thereby completing the proof of Theorem~\ref{thm:main_bright}.

The proof of Theorem~\ref{thm:main_dark} is considerably more complicated. Attempting to bifurcate directly from the black NLS soliton is obstructed by the fact that the origin lies in the essential spectrum of the linearization. To overcome this difficulty, we first set $\eps=0$ and restrict our attention to real-valued solutions of~\eqref{eq:LLE_main_tw_mod}. The advantage of this approach is that, in the formulation restricted to real-valued stationary solutions, $0$ is not in the essential spectrum of the linearization but instead corresponds to a simple isolated eigenvalue. Consequently, a Lyapunov--Schmidt reduction argument yields, provided $\eta > 0$ is sufficiently small, real-valued front and back solutions to~\eqref{eq:LLE_main_tw_mod} for $\eps = 0$ bifurcating from the black NLS soliton~\eqref{black_sol}. The resulting one-dimensional bifurcation equation $\mathcal{M}(\sigma)=0$ determines their phase shift.

We then construct a real-valued dark pulse solution to~\eqref{eq:LLE_main_tw_mod} for $\eps = 0$ by employing the framework developed in~\cite{Bengel2026Multiple} to concatenate a well-separated front and back solution with matching periodic end states. By carefully tracking the spatial Floquet exponents of these small-amplitude periodic end states under bifurcation, we show that the essential spectrum of the linearization about the dark pulse is bounded away from the origin for $\eta > 0$, provided that~\eqref{eq:F_mean} holds.

The key technical challenge in subsequently bifurcating with respect to $\eps$ is to establish the nondegeneracy of the real-valued dark pulse solution to~\eqref{eq:LLE_main_tw_mod} for $\eps = 0$, that is, the invertibility of the linearization about the dark pulse. The difficulty stems from the fact that the long periodic plateau, connecting the front and the back interface, possesses spatial Floquet exponents with identical real parts. As a consequence, the origin belongs to the \emph{absolute spectrum}~\cite{Sandstede2000} of the plateau state. Previous analyses considering traveling-pulse solutions which such long plateau states show that eigenvalues of the linearization about the pulse accumulate onto the absolute spectrum as the plateau length tends to infinity~\cite{Sandstede2000,Sandstede2000Gluing,Goh2022,Carter2021}. Inspired by~\cite{Sandstede2000,Sandstede2000Gluing}, we employ exponential dichotomies to reduce the eigenvalue problem at $\lambda=0$ to a boundary-value problem posed along the plateau. We solve this problem by a variation-of-constants argument, using Floquet theory and exploiting the exponential convergence of the front and back to their periodic end states. This analysis shows that the eigenvalue problem at $\lambda=0$ admits only the trivial solution, provided a suitable nonresonance condition in $\eta$ is satisfied. Since this condition fails only on a measure-zero set, nondegeneracy holds generically. To the best of the authors' knowledge, this is the first analysis of eigenvalue problems associated with pulse solutions possessing long \emph{periodic} plateau states with critical absolute spectrum. Previous works have treated either constant plateau states~\cite{Sandstede2000Gluing,Goh2022} or slowly varying ones~\cite{Carter2021}.

Having established the nondegeneracy of the real-valued dark pulse for fixed $\eta>0$ and $\eps=0$, the implicit function theorem provides a family of dark primary pulse solutions to~\eqref{eq:LLE_main_tw_mod} for sufficiently small $\eps>0$, yielding the proof of Theorem~\ref{thm:main_dark}. We refer to Figure~\ref{fig:dark_construction} for a schematic depiction of the main steps of the construction of the dark primary pulses.

Finally, exploiting the nondegeneracy of the bright and dark primary pulse solutions, we apply~\cite[Theorems~3.1 and~4.1]{Bengel2026Multiple} to concatenate and periodically extend any finite collection of primary pulses, thereby readily obtaining the existence results for soliton-based frequency combs in Corollaries~\ref{cor:bright_FC} and~\ref{cor:dark_FC}. The stability results in Corollary~\ref{cor:bright_FC} follow directly by applying~\cite[Corollaries~6.1 and~7.1 and Theorems~6.2 and~7.2]{Bengel2026Multiple}, which show that the frequency combs inherit the spectral stability properties of their constituent primary pulses.

\section{Existence of bright primary pulse solutions}\label{sec:existence_bright}

In this section, we establish the existence of stationary single-pulse solutions to~\eqref{eq:LLE_main} by bifurcating from the rotated and translated bright NLS soliton~\eqref{bright_sol}, which solves~\eqref{eq:LLE_main} for $\eps=0$. The pulse solutions are supported on a periodic background wave that accounts for the periodic forcing in the equation. In the next section, we prove that the primary pulse solutions are nondegenerate, in the sense that the linearization about the pulse is invertible. This nondegeneracy follows from the breaking of translational and rotational invariance in~\eqref{eq:LLE_main} for $\eps>0$ and allows us to apply~\cite[Theorems~3.1 and~4.1]{Bengel2026Multiple}, yielding the soliton-based frequency-comb solutions stated in Corollary~\ref{cor:bright_FC}.

The main result of this section reads as follows.

\begin{Proposition}[Existence of bright primary pulses] \label{prop:existence_1-pulse}
Fix parameters $\zeta,T > 0$ and $\omega \in \R$. Let $F \in C^1(\R,\C)$ be $T$-periodic. Assume that $(\theta_0,\sigma_0) \in \R^2$ obeys the existence condition $\mathcal{M}_{\theta_0,\sigma_0} = 4\sqrt{\zeta}$ and the nondegeneracy condition $D_{\theta_0,\sigma_0}(0) \neq 0$. Then, there exist constants $C,\eps_0> 0$ such that for all $\eps \in (-\eps_0,\eps_0)$ there exist stationary solutions $\unu_\eps \in H^2(\R) \oplus H_\per^2(0,T)$ and $v_\eps \in H_\per^2(0,T)$ to~\eqref{eq:LLE_main} satisfying~\eqref{eq:approximation_bright} and~\eqref{eq:approximation_background_wave}.
\end{Proposition}

Our strategy for proving Proposition~\ref{prop:existence_1-pulse} is as follows. We first construct a small-amplitude periodic background solution $v_\eps$ to 
\begin{align}\label{eq:LLE_main_tw}
 	-u'' + \zeta u - |u|^2 u + \iu \eps\left(\omega u' - u + F( x ) \right) = 0
\end{align}
by bifurcating from the zero rest state, which solves~\eqref{eq:LLE_main_tw} for $\eps = 0$. 
We then construct the bright primary pulse by substituting the ansatz
\begin{align} \label{eq:sol_ansatz}
    u = \phi_{\theta,\sigma} + v_\eps + w,
\end{align}
with $\phi_{\theta,\sigma} := \phi_\theta(\cdot-\sigma)$ into~\eqref{eq:LLE_main_tw}, where $w \in H^2(\R)$ is a small localized correction accounting for the fact that $\phi_{\theta,\sigma} + v_\eps$ is, in general, not an exact solution of~\eqref{eq:LLE_main_tw}. Finally, using a Lyapunov--Schmidt reduction argument, we show that, for sufficiently small $\eps > 0$, there exist $\theta_0,\sigma_0,w_\eps$, with $\|w_\eps\|_{H^2} \leq C \eps$ for some $\eps$-independent constant $C>0$, such that $
\unu_\eps = \phi_{\theta_0,\sigma_0} + v_\eps + w_\eps$ solves~\eqref{eq:LLE_main_tw}.

\paragraph*{Construction of periodic background wave.} We start with the construction of the periodic background wave $v_\eps$, which arises as a small stationary periodic solution of~\eqref{eq:LLE_main} through a bifurcation from the zero rest state.

\begin{Lemma}[Periodic background wave]\label{lem:small_periodic_wave}
There exist $C,\eps_0>0$ such that for all $\eps \in (-\eps_0,\eps_0)$ there exists a solution $v_\eps \in H_\per^2(0,T)$ to~\eqref{eq:LLE_main_tw} satisfying the bound~\eqref{eq:approximation_background_wave}. Moreover, $(-\eps_0,\eps_0) \to H^2_{\R,\per}(0,T), \eps \mapsto v_\eps$ is smooth.
\end{Lemma}

\begin{proof}
Define the function $\mathcal{F}\colon H_{\R,\per}^2(0,T) \times \R \to L_{\R,\per}^2(0,T)$ given by
\begin{align*}
	\mathcal{F}(v,\eps) = - v'' + \zeta v- |v|^2 v + \iu\eps  ( \omega v' - v +  F ).
\end{align*}
Then, $\mathcal{F}$ is a well-defined map since $v\in H_{\R,\per}^2(0,T)$ implies $|v|^2v \in H_{\R,\per}^2(0,T)$ and $F\in L_{\R,\per}^2(0,T)$. We observe that, $\mathcal{F}(0,0) = 0$ and $\mathcal{F}$ is continuously Fr\'echet differentiable with
\begin{align*}
	\partial_v \mathcal{F}(0,0) = - \partial_x^2 + \zeta .
\end{align*}
Since $\zeta> 0$, the operator $-\partial_x^2 + \zeta \colon H_{\R,\per}^2(0,T) \to L_{\R,\per}^2(0,T)$ is invertible with the bounded inverse $(- \partial_x^2 + \zeta )^{-1}\colon L_{\R,\per}^2(0,T) \to H_{\R,\per}^2(0,T)$ and, thus, the implicit function theorem yields $\eps_0>0$ and a smooth local branch of solutions $v\colon(-\eps_0,\eps_0) \to H_{\R,\per}^2(0,T),\eps \mapsto v_\eps$ such that $v_0=0$, solving the equation
\begin{align*}
	\mathcal{F}(v_\eps, \eps)=0
\end{align*}
uniquely in a neighborhood of $(0,0) \in H_{\R,\per}^2(0,T) \times \R$. Differentiating the equation $\mathcal{F}(v_\eps,\eps)=0$ with respect to $\eps$ and subsequently setting $\eps = 0$ yields
\begin{align*}
	\partial_v \mathcal{F}(0,0)\partial_\eps v_0 + \partial_\eps \mathcal{F}(0,0) = 0, 
\end{align*}
which is equivalent to 
\begin{align*}
	\partial_\eps v_0 = -\iu(- \partial_x^2 + \zeta )^{-1} F.
\end{align*}
By Taylor's theorem we infer the bound $\|v_\eps -\eps \partial_\eps v_0 \|_{H_\per^2(0,T)} \leq C\eps^2$ for some $\eps$-independent constant $C>0$, completing the proof.
\end{proof}

\paragraph*{Equation for the correction.} We substitute the ansatz~\eqref{eq:sol_ansatz} into~\eqref{eq:LLE_main_tw} with periodic background wave $v_\eps \in H_\per^2(0,T)$ from Lemma~\ref{lem:small_periodic_wave}. This results in an equation for the correction $w$ and the parameters $\theta,\sigma$ given by
\begin{align}\label{eq:correction}
	\mathfrak{L}_{\theta,\sigma}w + N(w,\eps,\theta,\sigma) = 0.
\end{align}
Here,
\begin{align*}
	\mathfrak{L}_{\theta,\sigma}\colon H_\R^2(\R) \to L_\R^2(\R),\qquad \mathfrak{L}_{\theta,\sigma} w = - w''+ \zeta w - 2|\phi_{\theta,\sigma}|^2 w- \phi_{\theta,\sigma}^2 \overline{w}
\end{align*}
is the linearization of the NLS equation~\eqref{NLS} about the bright soliton~\eqref{bright_sol} and
\begin{align*}
\begin{split}
	N(w,\eps,\theta,\sigma) = & - 2 |v_\eps+w|^2 \phi_{\theta,\sigma} -(v_\eps+w)^2 \overline{\phi_{\theta,\sigma}}
	-|v_\eps+ w|^2(v_\eps+w) + |v_\eps|^2 v_\eps \\
    &- 2|\phi_{\theta,\sigma}|^2 v_\eps
	- \phi_{\theta,\sigma}^2 \overline{v_\eps} + \iu \eps \left( \omega \phi_{\theta,\sigma}'+\omega w' - \phi_{\theta,\sigma} - w \right)
\end{split}
\end{align*}
denotes the residual. In particular, owing to Lemma~\ref{lem:small_periodic_wave}, there exists a constant $C>0$ such that $\|N(w,\eps,\theta,\sigma)\|_{L^2} \leq C (\|w\|_{L^2}^2 + |\eps|)$ for all $w \in H^2(\R)$ with $\|w\|_{H^1} \leq 1$, $\eps \in (-\eps_0,\eps_0)$, and $\theta,\sigma\in \R$.

\paragraph*{Lyapunov--Schmidt reduction argument.} We solve equation~\eqref{eq:correction} for the correction $w$ and the symmetry parameters $(\theta,\sigma)$ with the aid of a Lyapunov--Schmidt reduction argument. This relies on the following Fredholm properties of the linear operator $\mathfrak{L}_{\theta,\sigma}$, which were established in~\cite[Lemma 1]{Bengel2024Stability}.

\begin{Lemma}[Fredholm properties of linearization] \label{lem:NLS_linearization}
For every $(\theta,\sigma) \in \R^2$ the linear operator $\mathfrak{L}_{\theta,\sigma}$ is self-adjoint and Fredholm of index zero. Its kernel is spanned by $\phi'_{\theta,\sigma}$ and $\iu \phi_{\theta,\sigma}$.
\end{Lemma}

In order to solve~\eqref{eq:correction}, we impose two phase conditions on the correction term $w$ in~\eqref{eq:sol_ansatz}, which are of the form
\begin{align}\label{eq:phase_cond}
	\langle w, \phi'_{\theta,\sigma} \rangle_{L^2_\R} = 0, \quad\quad \langle w, \iu \phi_{\theta,\sigma} \rangle_{L^2_\R} = 0.
\end{align}
Note that these conditions are natural, since translations and rotations of the pulse can be absorbed by varying the free parameters $(\theta,\sigma)$. After introducing the orthogonal spectral projection onto $\ker (\mathfrak{L}_{\theta,\sigma})$, given by
\begin{align*}
    P_{\theta,\sigma} \colon H_\R^k(\R) \to H_\R^k(\R), \quad P_{\theta,\sigma} w = 
    \frac{\langle w, \phi'_{\theta,\sigma} \rangle_{L_\R^2}}{\|\phi'_{\theta,\sigma}\|_{L^2}^2}\phi'_{\theta,\sigma} +
    \frac{\langle w, \iu \phi_{\theta,\sigma} \rangle_{L_\R^2}}{\| \phi_{\theta,\sigma}\|_{L^2}^2} \iu \phi_{\theta,\sigma}, \qquad k \in \N_0,
\end{align*}
we can write \eqref{eq:phase_cond} in the compact form
\begin{align*}
	P_{\theta,\sigma} w = 0.
\end{align*}
The complementary projection $P_{\theta,\sigma}^\perp = I- P_{\theta,\sigma}$ satisfies $\mathrm{ran}(P_{\theta,\sigma}^\perp) = \ker(\mathfrak{L}_{\theta,\sigma})^\perp = \mathrm{ran}(\mathfrak{L}_{\theta,\sigma})$. 

We decompose~\eqref{eq:correction} into a non-singular part, a singular part, and the phase condition
\begin{align}
	P_{\theta,\sigma}^\perp \mathfrak{L}_{\theta,\sigma} P_{\theta,\sigma}^\perp w + P_{\theta,\sigma}^\perp N(w,\eps,\theta,\sigma) = 0, \label{eq:nonsing} \\
	P_{\theta,\sigma} N(w,\eps,\theta,\sigma)= 0,\label{eq:sing} \\
    P_{\theta,\sigma}w =0 \label{eq:phase_system}.
\end{align}
We start by solving the non-singular equation~\eqref{eq:nonsing} subject to the phase condition~\eqref{eq:phase_system}.

\begin{Lemma}\label{lem:nonsingular_equation}
Let $(\theta_0,\sigma_0) \in \R^2$. There exist $\eps_0,\delta_0>0$, and a smooth function 
$$
    w \colon (-\eps_0,\eps_0) \times (\theta_0-\delta_0,\theta_0+\delta_0)\times (\sigma_0-\delta_0,\sigma_0+\delta_0) \to H_\R^2(\R), \qquad (\eps,\theta,\sigma)\mapsto w(\eps,\theta,\sigma)
$$
that solves~\eqref{eq:nonsing} and~\eqref{eq:phase_system}. Moreover, we have $w(0,\theta,\sigma)= \partial_\theta w(0,\theta,\sigma) = \partial_\sigma w(0,\theta,\sigma) =0$ for all $(\theta,\sigma) \in (\theta_0-\delta_0,\theta_0+\delta_0)\times (\sigma_0-\delta_0,\sigma_0+\delta_0) $.
\end{Lemma}
\begin{proof}
Let $\eps_0 > 0$ be as in Lemma~\ref{lem:small_periodic_wave}.
Define the function $\mathcal{F}\colon H_\R^2(\R) \times (-\eps_0,\eps_0) \times \R^2 \to L_\R^2(\R)$ by
\begin{align*}
	\mathcal{F}(w,\eps,\theta,\sigma) = P_{\theta,\sigma}^\perp \mathfrak{L}_{\theta,\sigma} P_{\theta,\sigma}^\perp w + P_{\theta,\sigma}^\perp N(w,\eps,\theta,\sigma)+ P_{\theta,\sigma} w.
\end{align*}
Then, we have $\mathcal{F}(0,0,\theta_0,\sigma_0) = 0$, $\mathcal{F}$ is smooth in $(w,\eps,\theta,\sigma)$, and
\begin{align*}
	\partial_w \mathcal{F}(0,0,\theta_0,\sigma_0)\phi = P_{\theta_0,\sigma_0}^\perp \mathfrak{L}_{\theta_0,\sigma_0} P_{\theta_0,\sigma_0}^\perp \phi + P_{\theta_0,\sigma_0}\phi
\end{align*}
is an isomorphism from $H_\R^2(\R)$ onto $L_\R^2(\R)$ by Lemma~\ref{lem:NLS_linearization}, since $P_{\theta_0,\sigma_0}$ is the orthogonal spectral projection of $\mathfrak{L}_{\theta_0,\sigma_0}$ onto $\ker(\mathfrak{L}_{\theta_0,\sigma_0})$. Taking $\eps_0 > 0$ smaller if necessary, the implicit function theorem then provides $\delta_0>0$, an open neighborhood $W \subset H_\R^2(\R)$ of $0$, and a smooth branch of solutions $w\colon (-\eps_0,\eps_0) \times (\theta_0-\delta_0,\theta_0+\delta_0)\times (\sigma_0-\delta_0,\sigma_0+\delta_0) \to W, (\eps,\theta,\sigma) \mapsto w(\eps,\theta,\sigma)$ solving $\mathcal{F}(w,\eps,\theta,\sigma) = 0$ uniquely for $(w,\eps,\theta,\sigma) \in W\times (-\eps_0,\eps_0) \times (\theta_0-\delta_0,\theta_0+\delta_0)\times (\sigma_0-\delta_0,\sigma_0+\delta_0)$. Applying the projection $P_{\theta,\sigma}^\perp$ and $P_{\theta,\sigma}$ to $\mathcal{F}(w(\eps,\theta,\sigma),\eps,\theta,\sigma)=0$ yields
\begin{align*}
	P_{\theta,\sigma}^\perp \mathfrak{L}_{\theta,\sigma} P_{\theta,\sigma}^\perp w(\eps,\theta,\sigma)+P_{\theta,\sigma}^\perp N(w(\eps,\theta,\sigma),\eps,\theta,\sigma) = 0, \qquad
	P_{\theta,\sigma}w(\eps,\theta,\sigma) = 0.
\end{align*}
Finally, by the uniqueness of the solution we find $w(0,\theta,\sigma) =0$ for all $(\theta,\sigma) \in (\theta_0-\delta_0,\theta_0+\delta_0)\times (\sigma_0-\delta_0,\sigma_0+\delta_0)$, which further implies $\partial_\theta w(0,\theta,\sigma) = \partial_\sigma w(0,\theta,\sigma) =0$.
\end{proof}

\paragraph*{Solving the singular equation.} Let $(\theta_0,\sigma_0) \in \R^2$ be a solution to $\mathcal{M}_{\theta_0,\sigma_0} =4 \sqrt{\zeta} $. Assume that $D_{\theta_0,\sigma_0}(0) \neq 0$. Denote by $w(\eps,\theta,\sigma)$ the corresponding solution of~\eqref{eq:nonsing} subject to the phase condition~\eqref{eq:phase_system} obtained in Lemma~\ref{lem:nonsingular_equation}. Substituting $w(\eps,\theta,\sigma)$ into the singular equation~\eqref{eq:sing} leads to the two-dimensional bifurcation equation
\begin{align}\label{eq:bifurcation equation}
	P_{\theta,\sigma} N(w(\eps,\theta,\sigma),\eps,\theta,\sigma) = 0.
\end{align}
Define $g \colon (-\eps_0,\eps_0) \times (\theta_0-\delta_0,\theta_0+\delta_0) \times (\sigma_0-\delta_0,\sigma_0+\delta_0) \to \R^2$ by
$$
    g(\eps,\theta,\sigma) =
    \begin{pmatrix}
        \langle N(w(\eps,\theta,\sigma),\eps,\theta,\sigma) , \phi_{\theta,\sigma}'\rangle_{L_\R^2} \\
        \langle N(w(\eps,\theta,\sigma),\eps,\theta,\sigma) , \iu\phi_{\theta,\sigma}\rangle_{L_\R^2}
    \end{pmatrix}.
$$
Then, $g$ is smooth and the bifurcation equation~\eqref{eq:bifurcation equation} is equivalent to $g(\eps,\theta,\sigma) = 0$. Note that $g(0,\theta,\sigma) = 0$ for all $(\theta,\sigma) \in (\theta_0-\delta_0,\theta_0+\delta_0) \times (\sigma_0-\delta_0,\sigma_0+\delta_0)$ by Lemma~\ref{lem:nonsingular_equation}. Thus, we desingularize the bifurcation equation by defining the smooth function $\tilde g \colon (-\eps_0,\eps_0) \times (\theta_0-\delta_0,\theta_0+\delta_0) \times (\sigma_0-\delta_0,\sigma_0+\delta_0) \to \R^2$ by
$$
    \tilde{g}(\eps,\theta,\sigma) =
    \left\{\begin{array}{ll}
        \eps^{-1} g(\eps,\theta,\sigma), & \eps\neq 0,  \\
         \partial_\eps g(0,\theta,\sigma),& \eps = 0.
    \end{array}\right.
$$
Our goal is to obtain nontrivial solutions of the desingularized bifurcation equation $\tilde{g}(\eps,\theta,\sigma) = 0$. To this end, we compute $\tilde{g}(0,\theta_0,\sigma_0)$.

\begin{Lemma}[Analysis of desingularized bifurcation equation] \label{lem:bifurcation_function}
The function $\tilde{g}=(\tilde{g}_1,\tilde{g}_2)^\top$ satisfies
\begin{align*}
	\tilde{g}_1(0,\theta,\sigma) = \Im\langle \eu^{-\iu \theta} F'(\cdot +\sigma) ,\phi_{0,0} \rangle_{L_\R^2},\qquad
	\tilde{g}_2(0,\theta,\sigma) = -4 \sqrt{\zeta} + \Re \langle\eu^{-\iu \theta}F(\cdot+\sigma), \phi_{0,0}\rangle_{L_\R^2}.
\end{align*}    
\end{Lemma}

\begin{proof}
Since $v_0 = 0$ and $w(0,\theta,\sigma) =0$ by Lemmas~\ref{lem:small_periodic_wave} and~\ref{lem:nonsingular_equation}, we observe that
\begin{align*}
	g(\eps,\theta,\sigma) &=
	\begin{pmatrix}
		\eps\langle \iu\omega \phi_{\theta,\sigma}' - \iu \phi_{\theta,\sigma} - 2 |\phi_{\theta,\sigma}|^2 \partial_\eps v_0 - \phi_{\theta,\sigma}^2 \overline{\partial_\eps v_0}, \phi_{\theta,\sigma}' \rangle_{L_\R^2} +\check{g}_1(\eps,\theta,\sigma) \\
		\eps\langle \iu\omega \phi_{\theta,\sigma}' - \iu \phi_{\theta,\sigma} - 2 |\phi_{\theta,\sigma}|^2 \partial_\eps v_0 - \phi_{\theta,\sigma}^2 \overline{\partial_\eps v_0}, \iu \phi_{\theta,\sigma} \rangle_{L_\R^2} +
        \check{g}_2(\eps,\theta,\sigma)
	\end{pmatrix},
\end{align*}
where $\check{g} = (\check{g}_1,\check{g}_2)^\top$ satisfies $\check{g}(0,\theta,\sigma) = \partial_\eps \check{g} (0,\theta,\sigma) = 0$. Thus, it holds
\begin{align*}
	\tilde{g}_1(0,\theta,\sigma) &= \langle - 2 |\phi_{\theta,\sigma}|^2 \partial_\eps v_0 - \phi_{\theta,\sigma}^2 \overline{\partial_\eps v_0} , \phi_{\theta,\sigma}' \rangle_{L_\R^2}, \\
	\tilde{g}_2(0,\theta,\sigma) &= \langle - \iu \phi_{\theta,\sigma} - 2 |\phi_{\theta,\sigma}|^2 \partial_\eps v_0 - \phi_{\theta,\sigma}^2 \overline{\partial_\eps v_0} , \iu \phi_{\theta,\sigma} \rangle_{L_\R^2},
\end{align*}
and it remains to compute the two scalar products.

To compute $\tilde{g}_1(0,\theta,\sigma)$, we first note that $(-\partial_x^2+\zeta)^{-1}$ is self-adjoint. Moreover, inserting $\phi_{0,\sigma}$ into~\eqref{NLS} and differentiating, yields
$$
    -\phi_{0,\sigma}''' + \zeta \phi_{0,\sigma}' - 3 \phi_{0,\sigma}^2\phi_{0,\sigma}' = 0,
$$
which then implies that $(-\partial_x^2 + \zeta)^{-1}(\phi_{0,\sigma}^2\phi_{0,\sigma}') = \phi_{0,\sigma}'/3$. Hence, using~\eqref{eq:approximation_background_wave}, we obtain
\begin{align*}
	\tilde{g}_1(0,\theta,\sigma)  & 
    = \langle 2\iu\eu^{-\iu \theta} F  -  \iu \eu^{\iu \theta} \overline{F} , (-\partial_x^2 + \zeta)^{-1}(\phi_{0,\sigma}^2\phi_{0,\sigma}') \rangle_{L_\R^2} \\
    &= \frac{1}{3}\langle 2\iu\eu^{-\iu \theta} F - \iu \eu^{\iu \theta} \overline{F} ,\phi_{0,\sigma}' \rangle_{L_\R^2} \\
    &= -\frac{1}{3}\langle 2\iu\eu^{-\iu \theta} F'(\cdot +\sigma) - \iu \eu^{\iu \theta} \overline{F'(\cdot +\sigma)} ,\phi_{0,0} \rangle_{L_\R^2}.
\end{align*}
Combining the latter with
\begin{align*}
    \Re(2\iu\eu^{-\iu \theta} F'(\cdot +\sigma) - \iu \eu^{\iu \theta} \overline{F'(\cdot +\sigma)}) 
    = - 3 \Im(\eu^{-\iu \theta} F'(\cdot +\sigma)),
\end{align*}
we arrive at $\tilde{g}_1(0,\theta,\sigma) = \Im\langle \eu^{-\iu \theta} F'(\cdot +\sigma) ,\phi_{0,0} \rangle_{L_\R^2}.$

We proceed with computing $\tilde{g}_2(0,\theta,\sigma)$. Using~\eqref{eq:approximation_background_wave} and the fact that $\phi_{0,\sigma}$ solves~\eqref{NLS}, we infer
\begin{align*}
	\tilde{g}_2(0,\theta,\sigma)
    &=  - \|\phi_{0,0} \|_{L_\R^2}^2 + \langle 2\eu^{-\iu \theta}F- \eu^{\iu \theta}\overline{F} ,  (-\partial_x^2+\zeta)^{-1}(\phi_{0,\sigma}^3)\rangle_{L_\R^2}\\
    &=  - \|\phi_{0,0} \|_{L_\R^2}^2 + \langle 2\eu^{-\iu \theta}F- \eu^{\iu \theta}\overline{F} ,  \phi_{0,\sigma}\rangle_{L_\R^2}\\
    &=  - \|\phi_{0,0} \|_{L_\R^2}^2 + \langle 2\eu^{-\iu \theta}F(\cdot+\sigma)- \eu^{\iu \theta}\overline{F(\cdot+\sigma)} ,  \phi_{0,0}\rangle_{L_\R^2}.
\end{align*}
Combining the latter with $\|\phi_{0,0} \|_{L_\R^2}^2 = 4 \sqrt{\zeta}$ and
$$
    \Re(2\eu^{-\iu \theta}F(\cdot+\sigma)- \eu^{\iu \theta}\overline{F(\cdot+\sigma)})= \Re(\eu^{-\iu \theta}F(\cdot+\sigma)), 
$$
we obtain $\tilde{g}_2(0,\theta,\sigma) = -4 \sqrt{\zeta} + \Re \langle \eu^{-\iu \theta}F(\cdot+\sigma), \phi_{0,0}\rangle_{L_\R^2}$.
\end{proof}

We are now ready to solve the desingularized bifurcation equation with the aid of the implicit function theorem, thereby finishing the proof of Proposition~\ref{prop:existence_1-pulse}.

\begin{proof}[Proof of Proposition~\ref{prop:existence_1-pulse}.]
Let $\eps_0,\delta_0>0$ be as in Lemma~\ref{lem:nonsingular_equation}. Combining the formulas from Lemma~\ref{lem:bifurcation_function} with the fact that $(\theta_0,\sigma_0)\in \R^2$ satisfies the existence condition $\mathcal{M}_{\theta_0,\sigma_0} = 4\sqrt{\zeta}$ and the nondegeneracy condition $D_{\theta_0,\sigma_0}(0) \neq 0$, we obtain
\begin{align*}
	\tilde{g}(0,\theta_0,\sigma_0)=0, \qquad \det\left(\partial_{(\theta,\sigma)} \tilde{g}(0,\theta_0,\sigma_0)\right)\neq0.
\end{align*}
Therefore, the implicit function theorem provides $\eps_1 \in (0,\eps_0)$ and a smooth map 
$$
    (\theta,\sigma) \colon (-\eps_1,\eps_1) \to \R^2, \eps \mapsto (\theta(\eps), \sigma(\eps))
$$ 
with $(\theta(0),\sigma(0)) = (\theta_0,\sigma_0)$, solving the desingularized bifurcation equation $\tilde{g}(\eps,\theta(\eps),\sigma(\eps)) = 0$. We conclude that the triple
\begin{align*}
    (w(\eps,\theta(\eps),\sigma(\eps)),\theta(\eps), \sigma(\eps)) \in H_\R^2(\R) \times \R \times \R
\end{align*}
solves~\eqref{eq:correction} for all $\eps \in (-\eps_1,\eps_1)$. For $\eps \in (-\eps_0,\eps_1)$ we define $w_\eps = w(\eps,\theta(\eps),\sigma(\eps)) + \phi_{\theta(\eps),\sigma(\eps)} - \phi_{\theta_0,\sigma_0} \in H^2(\R)$ and $\unu_\eps = \phi_{\theta_0,\sigma_0} + v_\eps + w_\eps$. Then, $\unu_\eps \in H^2(\R) \oplus H_\per^2(0,T)$ is a solution of~\eqref{eq:LLE_main_tw}, which enjoys the bounds~\eqref{eq:approximation_bright} and~\eqref{eq:approximation_background_wave} by Lemmas~\ref{lem:small_periodic_wave} and~\ref{lem:nonsingular_equation}, thereby completing the proof.
\end{proof}

\section{Stability and nondegeneracy of bright pulse solutions}\label{sec:stability}

Fix parameters $\zeta,T > 0$ and $\omega \in \R$. Let $F \in C^1(\R,\C)$ be $T$-periodic. Let $(\theta_0,\sigma_0) \in \R^2$ be a solution to $\mathcal{M}_{\theta_0,\sigma_0} = 4 \sqrt{\zeta}$ and assume that $D_{\theta_0,\sigma_0}(0) \neq 0$. Let $\unu_\eps \in H^2(\R) \oplus H_\per^2(0,T)$ be a bright single-pulse solution to~\eqref{eq:LLE_main_tw}, whose existence was established in Proposition~\ref{prop:existence_1-pulse}. 

In this section, we analyze the spectrum of the linearization $\El_\eps(\ubu_\eps)-\eps \colon H^2(\R) \to L^2(\R)$ of~\eqref{eq:LLE_main_systems} about $\ubu_\eps = (\Re(\unu_\eps),\Im(\unu_\eps))^\top$. We first establish that the essential spectrum of $\El_\eps(\ubu_\eps)-\eps$ is confined to the open left-half plane. Using Krein index theory, we then prove that spectral stability is decided by the eigenvalues of $\El_\eps(\ubu_\eps)-\eps$ bifurcating from the isolated zero eigenvalue of algebraic multiplicity four associated with the bright NLS soliton. A Lyapunov--Schmidt reduction argument shows that that the leading-order location of these eigenvalues is determined by the zeros of the reduced Evans function $D_{\theta_0,\sigma_0}$.

In particular, our spectral analysis establishes that the bright primary pulse is nondegenerate, i.e., the linearization $\El_\eps(\ubu_\eps)-\eps$ is invertible. We can therefore apply the results of~\cite{Bengel2026Multiple} to concatenate and periodically extend any finite collection of primary pulses. This yields the soliton-based frequency-comb solutions from  Corollary~\ref{cor:bright_FC}, together with their spectral stability properties.  

We start by stating a spectral a-priori bound, which will be used throughout the spectral stability analysis. The result follows along the lines of~\cite[Lemma~2]{Bengel2025Existence} and therefore the proof is omitted.

\begin{Lemma}[Spectal a-priori bound]\label{lem:spectral_aprior_bounds}
Let $\rho > 0$ and $\eps \in \R$. There exist constants $\eta_1,\eta_2>0$ such that for all $\ubu \in L^\infty(\R,\R^2)$ with $\|\ubu\|_{L^\infty} \leq \rho$ we have
$$
    \big(\{\lambda \in \C : |\Re(\lambda)| \geq \eta_1\} \cup \{\lambda \in \C : |\Im(\lambda)| \geq \eta_2, \Re(\lambda) \neq 0\} \big) \cap \sigma(\El_\eps(\ubu)) = \emptyset.
$$
\end{Lemma}

For $\eps=0$, the linear operator coincides with the linearization of the NLS equation~\eqref{NLS} about the bright soliton~\eqref{bright_sol}. We exploit the fact that the spectrum and (generalized) eigenfunctions of this operator are known explicitly~\cite{KapitulaPromislow2013}.

\begin{Lemma}[Spectrum of bright NLS soliton]\label{lem:spec_NLS_bright}
We have
$$
    \sigma(\El_0(\ubu_0)) = \{0\} \cup (-\iu \infty,-\iu \zeta] \cup [\iu \zeta,+\iu\infty),
$$
and the isolated eigenvalue $\lambda=0$ has geometric multiplicity two and algebraic multiplicity four. The associated (generalized) eigenfunctions are given by
\begin{align}\label{eq:Jordan_chain}
    \mathbf{p}_0 = \boldsymbol{\phi}'_{\theta_0,\sigma_0}, \qquad 
    \mathbf{p}_1 = \frac{x}{2}J \boldsymbol{\phi}_{\theta_0,\sigma_0}, \qquad
    \mathbf{q}_0 = -J\boldsymbol{\phi}_{\theta_0,\sigma_0}, \qquad 
    \mathbf{q}_1 =\partial_\zeta\boldsymbol{\phi}_{\theta_0,\sigma_0}
\end{align}
and obey the relations
\begin{align*}
    \El_0(\ubu_0) \mathbf{p}_0 = 0, \qquad \El_0(\ubu_0) \mathbf{p}_1 = \mathbf{p}_0, \qquad
    \El_0(\ubu_0) \mathbf{q}_0 = 0, \qquad \El_0(\ubu_0) \mathbf{q}_1 = \mathbf{q}_0.
\end{align*}
\end{Lemma}

\paragraph*{Essential spectrum.} Using standard spectral perturbation arguments, we show that the essential spectrum of $\El_\eps(\ubu_\eps)-\eps$ is confined to the open left-half plane admitting a spectral gap of size $\eps/2$.

\begin{Lemma}[Analysis of essential spectrum]\label{lem:essential_spectrum}
Provided $\eps>0$ is sufficiently small, we have
$$
    \sigma_\textup{ess}(\El_\eps(\ubu_\eps)-\eps) \subset\left\{\lambda \in \C : \Re(\lambda) \leq -\tfrac{\eps}{2}\right\}.
$$
\end{Lemma}
\begin{proof}
Let $\vb_\eps = (\Re(v_\eps),\Im(v_\eps))^\top$. By~\cite[Lemma~C.1]{Bengel2026Multiple} the multiplication operator $\El_\eps(\ubu_\eps) - \El_\eps(\vb_\eps) \colon H^2(\R) \to L^2(\R)$ is compact. Hence, using Weyl's theorem, cf.~\cite[Theorem VI.5.26]{Kato1995}, and Floquet theory, cf.~\cite[Section~2.1.3]{KapitulaPromislow2013}, we arrive at
\begin{align} \label{eq:spec_eq}
     \sigma_\text{ess}(\El_\eps(\ubu_\eps)) = \sigma_\text{ess}(\El_\eps(\vb_\eps)) = \sigma(\El_\eps(\vb_\eps)).
\end{align}
We decompose
\begin{align*}
    \El_\eps(\vb_\eps) = \El_\eps(0) +  \El_\eps(\vb_\eps) - \El_\eps(0).
\end{align*}
Using the continuous embedding $H_\per^1(0,T) \hookrightarrow L^\infty(0,T)$ in combination with the bounds of Lemma~\ref{lem:small_periodic_wave}, we find $\eps$-independent constants $C_{1,2}>0$ such that, provided $\eps > 0$ is sufficiently small, we have 
$$
\|v_\eps\|_{L^\infty} \leq C_1 \eps, \qquad    \|\El_\eps(\vb_\eps) - \El_\eps(0)\|_{L^2 \to L^2} \leq C_2 \eps^2.
$$
On the other hand, since $\El_\eps(0)$ is skew-adjoint, Stone's theorem, cf.~\cite[Theorem 3.24]{EngelNagel2000}, yields
\begin{align*}
    \|(\El_\eps(0) - \lambda)^{-1}\|_{L^2 \to L^2} \leq \frac{1}{|\Re(\lambda)|}, \qquad \lambda \not\in \iu \R.
\end{align*}
In particular, for $|\Re(\lambda)|>|\eps|/2$ we obtain
\begin{align*}
    \El_\eps(\vb_\eps)-\lambda = \left(I+ (\El_\eps(\vb_\eps)-\El_\eps(0))(\El_\eps(0)-\lambda)^{-1}\right) (\El_\eps(0)-\lambda)
\end{align*}
by employing the bound
\begin{align*}
    \|(\El_\eps(\vb_\eps)-\El_\eps(0))(\El_\eps(0)-\lambda)^{-1}\|_{L^2\to L^2} \leq 2C_2 \eps.
\end{align*}
Therefore, provided $0\leq\eps< 1/(2C_2)$, the operator $\El_\eps(\vb_\eps)-\lambda$ is invertible for all $\lambda \in \C$ with $|\Re(\lambda)| >|\eps|/2$. Combining the latter with~\eqref{eq:spec_eq} completes the proof. 
\end{proof}

\paragraph*{Point spectrum.} We now turn our attention to the point spectrum of the operator $\El_\eps(\ubu_\eps)-\eps$. We first omit the damping term and analyze the point spectrum of the Hamiltonian linear operator $\El_\eps(\ubu_\eps) = J L_\eps(\ubu_\eps)$. We later reintroduce the damping, which shifts the entire point spectrum by $-\eps$.

A key aspect of understanding the point spectrum of $\El_\eps(\ubu_\eps)$ is analyzing the splitting of the isolated eigenvalue $\lambda = 0$ of multiplicity four of $\El_0(\ubu_0)$ as $\eps > 0$; see Lemma~\ref{lem:spec_NLS_bright}. We will show that this eigenvalue generically splits into four simple eigenvalues and determine their leading-order position. Notably, the splitting behavior depends on the system's parameters of~\eqref{eq:LLE_main_tw}, as well as the rotation $\theta_0$ and the position $\sigma_0$ of the underlying bright soliton~\eqref{bright_sol}. We follow the approach in~\cite[Theorem 4.1]{KapitulaKevrekidis2004} and~\cite[Lemma 7.2.8]{KapitulaPromislow2013}. That is, we expand the solution $\unu_\eps$ in powers of $\eps$ which yields
$$
    \unu_\eps = u_0 + \eps u_1 + \mathcal{O}(\eps^2),
$$
where
\begin{align}\label{eq:expansion_u}
    u_0 = \phi_{\theta_0,\sigma_0}, \qquad
    u_1 = \iu \theta'(0)u_0 + \sigma'(0)u_0' + v_1 + w_1, \qquad
    v_1 = \partial_\eps v_0, \qquad w_1 = \partial_\eps w(0,\theta_0,\sigma_0),
\end{align}
and $w(\eps,\theta,\sigma)$ denotes the map from Lemma~\ref{lem:nonsingular_equation}. This leads to an expansion of the operator $L_\eps(\ubu_\eps) = L_0 + \eps L_1 + \mathcal{O}(\eps^2)$, with $L_0 = L_0(\ubu_0)$ and
$$
    L_1 = - 2
    \begin{pmatrix}
        3\Re(u_0)\Re(u_1) + \Im(u_0) \Im(u_1) &\Re(u_0)\Im(u_1)+ \Re(u_1) \Im(u_0)\\
        \Re(u_0)\Im(u_1)+ \Re(u_1) \Im(u_0)  & \Re(u_0)\Re(u_1) + 3\Im(u_0) \Im(u_1)
    \end{pmatrix} - \omega J \partial_x.
$$
As is described in~\cite[Theorem~4.1]{KapitulaKevrekidis2004}, the matrices
\begin{align}\label{eq:matrices_reduces_evp}
    \mathcal{A} = 
    \begin{pmatrix}
        \langle L_1 \mathbf{p}_0,\mathbf{p}_0\rangle_{L^2} &
        \langle L_1 \mathbf{p}_0,\mathbf{q}_0\rangle_{L^2} \\
        \langle L_1 \mathbf{q}_0,\mathbf{p}_0\rangle_{L^2} &
        \langle L_1 \mathbf{q}_0,\mathbf{q}_0\rangle_{L^2} 
    \end{pmatrix}, \qquad
    \mathcal{B} =
    \begin{pmatrix}
        \langle J\mathbf{p}_1,\mathbf{p}_0\rangle_{L^2} &
        \langle J\mathbf{p}_1,\mathbf{q}_0\rangle_{L^2} \\
        \langle J\mathbf{q}_1,\mathbf{p}_0\rangle_{L^2} &
        \langle J\mathbf{q}_1,\mathbf{q}_0\rangle_{L^2} 
    \end{pmatrix}
\end{align}
then allow us to trace the splitting of the multiplicity-four zero eigenvalue. Here, $\mathbf{p}_{0,1}$ and $\mathbf{q}_{0,1}$ are the (generalized) eigenfunctions of $\El_0$; see Lemma~\ref{lem:spec_NLS_bright}. For the reader's convenience, we recall the corresponding result from~\cite[Theorem~4.1]{KapitulaKevrekidis2004}. Its proof relies on regular perturbation theory and Lyapunov--Schmidt reduction.

\begin{Theorem}[Splitting of zero eigenvalue]\label{thm:eigenvalue_splitting}
There exists $\delta>0$ such that, provided $\eps>0$ is sufficiently small, the set $
B_\delta(0) \cap \sigma(\El_\eps(\ubu_\eps))$
consist of the four eigenvalues
$$
    \lambda = \eps^{1/2}\lambda_1  + \mathcal{O}(\eps)
$$
with eigenvectors $\ub = \left(\mathbf{p}_0 + \eps^{1/2} \lambda_1 \mathbf{p}_1\right) v_1 + \left(\mathbf{q}_0 + \eps^{1/2} \lambda_1 \mathbf{q}_1\right) v_2 + \mathcal{O}(\eps)$, where $\lambda_1 \in \C$ and $\vb = (v_1,v_2)^\top \in \C^2$ solve the bifurcation equation
\begin{align} \label{eq:bif_eq_spec}
    \mathcal{A} \vb = - \lambda_1^2 \mathcal{B} \vb.
\end{align}
The Krein signature of a simple, nonzero, purely imaginary eigenvalue $\lambda = \eps^{1/2}\lambda_1  + \mathcal{O}(\eps)$ is given by $\sign \langle \mathcal{A}\vb,\vb\rangle_{\C^2}$.
\end{Theorem}

The next lemma provides explicit expressions for the matrices $\mathcal{A}$ and $\mathcal{B}$ in the bifurcation equation~\eqref{eq:bif_eq_spec}. We relegate the calculations to Appendix~\ref{app:computations}.

\begin{Lemma}[Analysis of bifurcation equation] \label{lem:formula_for_matrices_AB}
It holds $\mathcal{A} = A_{\theta_0,\sigma_0}$ and $\mathcal{B} = B$, where $A_{\theta,\sigma}$ and $B$ are given by~\eqref{eq:explicit_matrices}. 
\end{Lemma}

According to Theorem~\ref{thm:eigenvalue_splitting} and Lemma~\ref{lem:formula_for_matrices_AB}, roots of the reduced Evans function $D_{\theta_0,\sigma_0}$ correspond to the leading-order coefficients of the eigenvalues of $\El_\eps(\ubu_\eps)$ bifurcating from the zero eigenvalue. As $D_{\theta_0,\sigma_0}(0) \neq 0$, this establishes nondegeneracy of the bright primary pulse $\unu_\eps$. Moreover, a Krein-index counting argument~\cite{KapitulaKevrekidis2004,AddendumKapitulaKevrekidis2004} shows that its spectral stability is decided by the eigenvalues bifurcating from $0$. This leads to the following result.

\begin{Proposition}[Nondegeneracy and spectral stability]\label{prop:spec_stability_1-pulse} Provided $\eps>0$ is sufficiently small, the following assertions hold:
\begin{itemize}
    \item[(i)] The operator $\El_\eps(\ubu_\eps) - \eps$ is invertible.
    \item[(ii)] If $D_{\theta_0,\sigma_0}$ has four distinct purely imaginary roots, then there exists $\mu \in (0,\eps)$ such that $\sigma(\El_\eps(\ubu_\eps) - \eps) \subset \{\lambda \in \C : \Re(\lambda) \leq -\mu\}$. 
    \item[(iii)] If $D_{\theta_0,\sigma_0}$ has a root in $\C \setminus \iu\R$, then there exists $\lambda \in \sigma(\El_\eps(\ubu_\eps) -\eps)$ with $\Re(\lambda) > 0$.
\end{itemize}
\end{Proposition}
\begin{proof}
First, we observe that $D_{\theta_0,\sigma_0}(0) \neq 0$ implies, provided $\eps > 0$ is sufficiently small, that $0 \not \in \sigma(\El_\eps(\ubu_\eps)-\eps)$ by Theorem~\ref{thm:eigenvalue_splitting} and Lemma~\ref{lem:formula_for_matrices_AB}. This yields the first assertion. 

To prove the second assertion, we first note the essential spectrum of $\El_\eps(\ubu_\eps)-\eps$ is confined to the open left-half plane with spectral gap $\eps/2$ by Lemma~\ref{lem:essential_spectrum}. To analyze the point spectrum of $\El_\eps(\ubu_\eps)$, we employ the Krein index formula~\cite[Theorem 1]{AddendumKapitulaKevrekidis2004}, which reads
\begin{align}\label{eq:Kreinindex_formula}
    k_r(\El_\eps(\ubu_\eps)) + 2 k_i^-(\El_\eps(\ubu_\eps)) + 2 k_c(\El_\eps(\ubu_\eps)) = n(L_\eps(\ubu_\eps)),
\end{align}
where $k_r(\El_\eps(\ubu_\eps))$ is the number of positive real unstable eigenvalues of $\El_\eps(\ubu_\eps)$, $k_i^-(\El_\eps(\ubu_\eps))$ is the number of purely imaginary eigenvalues with positive imaginary part and negative Krein signature, and $k_c(\El_\eps(\ubu_\eps))$ is the number of eigenvalues with positive real and imaginary part, all counted with algebraic multiplicities. Finally, $n(L_\eps(\ubu_\eps))$ is the number of negative eigenvalues of $L_\eps(\ubu_\eps)$, counting multiplicity. 

We show that $k_r(\El_\eps(\ubu_\eps)) = k_c(\El_\eps(\ubu_\eps)) = 0$. To this end, we first establish an upper bound on $n(L_\eps(\ubu_\eps))$ by applying Sturm--Liouville theory, cf.~\cite[Section~2.3]{KapitulaPromislow2013}. Thus, using the fact that the bright soliton~\eqref{bright_sol} has no zeros and precisely one extremum, we obtain that $n(L_0(\ubu_0)) = 1$ and $z(L_0(\ubu_0)) =2$ by Lemma~\ref{lem:spec_NLS_bright}, where $z(L_0(\ubu_0))$ denotes the multiplicity of the zero eigenvalue of $L_0(\ubu_0)$. Regular perturbation theory~\cite[Section~VII.3.2]{Kato1995} then yields, provided $\eps > 0$ is sufficiently small, that $1 \leq n(L_\eps(\ubu_\eps)) \leq 3$. From Theorem~\ref{thm:eigenvalue_splitting}, Lemma~\ref{lem:formula_for_matrices_AB}, and $D_{\theta_0,\sigma_0}(\lambda) \neq 0$ for all $\lambda\in \R$, we obtain $\delta>0$ such that, provided $\eps > 0$ is sufficiently small, we have $(-\delta,\delta) \cap \sigma(\El_\eps(\ubu_\eps)) = \emptyset$. On the other hand, Lemma~\ref{lem:spectral_aprior_bounds} yields an $\eps$-independent constant $C>0$ such that for all $\lambda \in \C$ with $|\Re(\lambda)| \geq C$ we have $\lambda \not\in \sigma(\El_\eps(\ubu_\eps))$. Define the compact set $\mathcal{K}:= [-C,-\delta] \cup [\delta,C]$. By Lemma~\ref{lem:spec_NLS_bright} we have $\sigma(\El_0(\ubu_0)) \cap \mathcal{K} = \emptyset$. A standard perturbation argument based on Neumann series~\cite{Kato1995} then yields, provided $\eps > 0$ is sufficiently small, that $\sigma(\El_\eps(\ubu_\eps)) \cap \mathcal{K} = \emptyset$. Thus, we establish $\sigma(\El_\eps(\ubu_\eps)) \cap \R = \emptyset$ and hence $k_r(\El_\eps(\ubu_\eps)) = 0$. The index formula~\eqref{eq:Kreinindex_formula} then yields that $n(L_\eps(\ubu_\eps))$ is an even number, implying $n(L_\eps(\ubu_\eps)) = 2$.

It remains to prove that $k_i^-(\El_\eps(\ubu_\eps)) = 1$, which then implies $k_c(\El_\eps(\ubu_\eps)) = 0$ by the index formula~\eqref{eq:Kreinindex_formula} as desired.
Let $\pm\iu \smash{\omega_{1/2}} \in \iu\R$ with $\smash{\omega_{1/2}}>0$ be the distinct roots of the function $D_{\theta_0,\sigma_0}$. Let $\vb^{1/2} \in \C^{2}$ be the eigenvectors to the generalized eigenvalue problem
$$
    A_{\theta_0,\sigma_0} \vb = \omega_{1/2}^2 B \vb.
$$
We have to show that one of the scalar products $\langle A_{\theta_0,\sigma_0}\vb^{1/2},\vb^{1/2} \rangle_{\C^2}$ has a negative sign, which then implies by Theorem~\ref{thm:eigenvalue_splitting} and Lemma~\ref{lem:formula_for_matrices_AB} that the corresponding eigenvalue $\lambda = \iu \omega_{1/2} \eps^{1/2} + \mathcal{O}(\eps) \in \iu\R$ of $\El_\eps(\ubu_\eps)$ has negative Krein signature. Note that, since $D_{\theta_0,\sigma_0}$ has four distinct roots, the $\mathcal{O}(\eps)$-remainder in the formula of $\lambda$ is purely imaginary because of the spectral symmetry of the Hamiltonian operator $\El_\eps(\ubu_\eps)$, cf.~\cite[Proposition 5.1.2]{KapitulaPromislow2013}.

We define
$$
    V_{1/2,\pm} =
    \begin{pmatrix}
        \vb^{1/2} \\ \pm \iu \omega_{1/2} B \vb^{1/2}
    \end{pmatrix} \in \C^4, \qquad 
    M =
    \begin{pmatrix}
        A_{\theta_0,\sigma_0} & 0 \\ 0 & B^{-1}
    \end{pmatrix} \in \C^{4\times 4},
$$
where we note that $M^* = M$, and obtain
$$
    \tilde JM V_{1/2,\pm } = \pm \iu \omega_{1/2} V_{1/2,\pm}, \qquad \tilde J =
    \begin{pmatrix}
        0 & \mathrm{Id}_{\C^2} \\
        - \mathrm{Id}_{\C^2}& 0
    \end{pmatrix}
    .
$$
Since $\sigma(\tilde JM) = \{\pm \iu \omega_{1/2}\} \subset \iu \R\setminus\{0\}$ and $1 = n(B^{-1}) \leq n(M)$, we infer from the Krein index formula~\cite[Theorem 7.1.9]{KapitulaPromislow2013} for the Hamiltonian matrix $\tilde JM$ and $k_r(\tilde JM) = k_c(\tilde JM) = 0$ that $n(M)\geq 2$ and $k_i^-(\tilde JM)\geq 1$. Without loss of generality, we can assume that the eigenvalue $\iu \omega_1$ has a negative Krein index. Then, it holds
\begin{align*}
    0>\langle M V_{1,+},V_{1,+}\rangle_{\C^4}
    &= \langle A_{\theta_0,\sigma_0} \vb^1,\vb^1\rangle_{\C^2} + \langle B^{-1} \iu \omega_{1} B \vb^{1},\iu \omega_{1} B \vb^{1}\rangle_{\C^2} \\
    &= \langle A_{\theta_0,\sigma_0} \vb^1,\vb^1\rangle_{\C^2} + \langle \omega_1^2 B \vb^{1}, \vb^{1}\rangle_{\C^2}
    = 2 \langle A_{\theta_0,\sigma_0} \vb^1,\vb^1\rangle_{\C^2}.
\end{align*}
Thus, we arrive at $\langle A_{\theta_0,\sigma_0} \vb^1,\vb^1\rangle_{\C^2}<0$ and, hence, $k_i^- (\El_\eps(\ubu_\eps)) \geq 1$, which, in turn, implies $k_i^- (\El_\eps(\ubu_\eps)) = 1$. This establishes assertion~(ii) upon shifting the spectrum by $-\eps$.

Finally, we prove the assertion (iii). By assumption, there exists a root $\lambda_1 \in \C$ of $D_{\theta_0,\sigma_0}$ with $\Re(\lambda_1) \neq 0$. Theorem~\ref{thm:eigenvalue_splitting} and Lemma~\ref{lem:formula_for_matrices_AB} yield a corresponding eigenvalue $\lambda^{(1)} = \eps^{1/2} \lambda_1 + \mathcal{O}(\eps)$ of the operator $\El_\eps(\ubu_\eps)$. By spectral symmetry of the Hamiltonian operator $\El_\eps(\ubu_\eps)$, we find that $\lambda^{(2)} = |\Re(\lambda^{(1)})| + \iu \Im(\lambda^{(1)}) = \eps^{1/2} \left(|\Re(\lambda_1)| + \iu \Im(\lambda_1)\right) + \mathcal{O}(\eps)$ is also an eigenvalue of $\El_\eps(\ubu_\eps)$. Therefore, $\lambda^{(2)}-\eps \in \sigma(\El_\eps(\ubu_\eps)-\eps)$ is an eigenvalue with strictly positive real part. 
\end{proof}

\paragraph*{Proofs of Theorem~\ref{thm:main_bright} and Corollary~\ref{cor:bright_FC}.} The proof of Theorem~\ref{thm:main_bright} follows directly by combining Propositions~\ref{prop:existence_1-pulse} and~\ref{prop:spec_stability_1-pulse}. In addition, Proposition~\ref{prop:spec_stability_1-pulse} implies that the bright primary pulses constructed in Proposition~\ref{prop:existence_1-pulse} are nondegenerate. We may therefore invoke~\cite[Theorem~3.1]{Bengel2026Multiple} to construct stationary $M$-pulse solutions to~\eqref{eq:LLE_main}, arising near the formal concatenation of $M$ primary pulses. By~\cite[Lemma~3.3, Corollary~6.1, and Theorem~6.2]{Bengel2026Multiple}, together with Lemma~\ref{lem:spectral_aprior_bounds}, these $M$-pulse solutions are nondegenerate and inherit the spectral stability properties of their constituent primary pulses. We then apply~\cite[Theorem~4.1]{Bengel2026Multiple} to periodically extend the $M$-pulse solutions. Finally, Lemma~\ref{lem:spectral_aprior_bounds} and~\cite[Corollary~7.1 and Theorem~7.2]{Bengel2026Multiple} yield that the resulting  stationary periodic $M$-pulse solutions to~\eqref{eq:LLE_main} inherit the spectral stability properties of the underlying $M$-pulse. This completes the proof of Corollary~\ref{cor:bright_FC}.

\section{Existence of dark pulse solutions}\label{sec:existence_black}

Fix parameters $\zeta,T>0$ and $\omega \in \R$, and let $F \in C^1(\R,\R)$ be $T$-periodic. In this section, we construct stationary dark primary pulse solutions to~\eqref{eq:LLE_main_tw_mod}. These solutions arise from formal concatenations of the black NLS soliton~\eqref{black_sol}, which solves~\eqref{eq:LLE_main_tw_mod} for $\eps = \eta = 0$. We show that the resulting dark primary pulse solutions are nondegenerate, in the sense that the linearization about the pulse is invertible. This nondegeneracy allows us to apply results from~\cite{Bengel2026Multiple} to construct the dark soliton-based frequency-comb solutions stated in Corollary~\ref{cor:dark_FC}. 

We rewrite~\eqref{eq:LLE_main_tw_mod} as a system 
\begin{align}\label{eq:LLE_main_tw_mod_sys}
    \ub_t = J(\ub_{xx} + \zeta \ub - |\ub|^2 \ub) + \eps \left(\omega \ub_x -\ub\right) + \eta J\mathbf{F}(x)
\end{align}
in $\ub = (\Re(u),\Im(u))^\top$, where $J$ and $\mathbf{F}$ are defined in~\eqref{eq:defJF}. The linearization $\El_\eps(\ubu) - \eps \colon H^2(\R) \to L^2(\R)$ of~\eqref{eq:LLE_main_tw_mod_sys} about a stationary solution $\ubu = (\unu,0)^\top$ is then given by
\begin{align} \label{eq:defL_eps}
	L_\eps(\ubu) = \begin{pmatrix} L_+(\unu) & 0 \\ 0 & L_-(\unu)\end{pmatrix} - \eps \omega J \partial_x, \qquad \El_\eps(\ubu) = JL_\eps(\ubu),
\end{align}
where $L_\pm(\unu) \colon H^2(\R) \to L^2(\R)$ are the second-order operators
\begin{align} \label{eq:defL+L_}
L_+(\unu) = \partial_x^2 + \zeta - 3\unu^2, \qquad L_-(\unu) = \partial_x^2 + \zeta - \unu^2.
\end{align}

Following the strategy outlined in~\S\ref{sec:method_of_proof} and Figure~\ref{fig:dark_construction}, we first set $\eps = 0$ and restrict attention to real-valued solutions of~\eqref{eq:LLE_main_tw_mod}. Bifurcating from the black NLS soliton~\eqref{black_sol}, we construct nondegenerate real-valued front and back solutions connecting periodic end states. We then apply the results of~\cite{Bengel2026Multiple} to concatenate a front and a back solution, thereby obtaining a nondegenerate real-valued dark primary pulse solution to~\eqref{eq:LLE_main_tw_mod} for $\eps = 0$. Finally, an application of the implicit function theorem with respect to $\eps$ yields nondegenerate dark primary pulse solutions to~\eqref{eq:LLE_main_tw_mod} for $\eps,\eta > 0$.

\begin{figure}[tb]
    \centering
    \includegraphics[width=\textwidth]{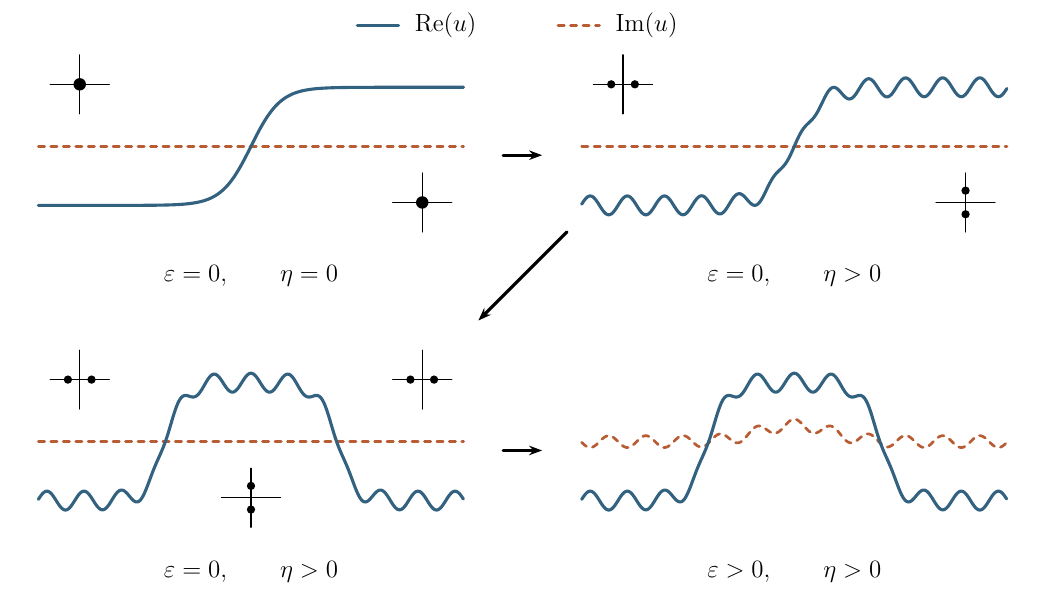}
    \caption{Upper-left panel: the black soliton, or domain wall, solving~\eqref{eq:LLE_main_tw_mod} for $\eps=\eta=0$. Upper-right panel: bifurcating fronts with periodic end states, solving~\eqref{eq:LLE_main_tw_mod} for $\eps=0$ and $\eta>0$. Lower-left panel: a dark primary pulse solving~\eqref{eq:LLE_main_tw_mod} for $\eps=0$ and $\eta>0$, arising near the formal concatenation of a well-separated front and back solution with matching periodic end states. Lower-right panel: bifurcating primary pulses solving~\eqref{eq:LLE_main_tw_mod} for $\eps,\eta>0$. The insets depict the small spatial Floquet exponents~\cite{KapitulaPromislow2013} of the periodic end and plateau states as solutions to~\eqref{eq:LLE_main_tw_mod}. In particular, the origin is in the absolute spectrum~\cite{Sandstede2000} of the long periodic plateau state of the primary dark pulse.  
    }
    \label{fig:dark_construction}
\end{figure}

\paragraph*{Front and back solutions with periodic end states.} Real-valued stationary solutions to~\eqref{eq:LLE_main_tw} for $\eps = 0$ satisfy the equation
\begin{align}\label{eq:LLE_defocusing_stationary_real}
    u'' + \zeta u - u^3 + \eta F(x) = 0.
\end{align}
In the following, we call a solution $\unu \colon \R \to \R$  to~\eqref{eq:LLE_defocusing_stationary_real} \emph{nondegenerate} if the operator $L_+(\unu)$, given by~\eqref{eq:defL+L_}, is invertible.

We observe that $\psi_{+,\sigma} := \psi(\cdot + \sigma)$, where $\psi$ is the black NLS soliton~\eqref{black_sol2}, is a front solution to~\eqref{eq:LLE_defocusing_stationary_real} for $\eta = 0$ connecting $-\sqrt{\zeta}$ to $\sqrt{\zeta}$ for any $\sigma \in \R$. By  symmetry, $\psi_{-,\sigma} := \psi_{\sigma}(-\cdot)$ is a back solution to~\eqref{eq:LLE_defocusing_stationary_real} for $\eta = 0$ connecting $\sqrt{\zeta}$ to $-\sqrt{\zeta}$. The next result shows that these front and back solutions persist for small values of $\eta > 0$ and that the bifurcating front and backs are nondegenerate as solutions to~\eqref{eq:LLE_defocusing_stationary_real}, connecting small periodic end states $v_{\pm}\in H_\per^2(0,T)$ around $\pm\sqrt{\zeta}$. Its proof is analogous to~\cite[Theorem 8.2]{Bengel2026Multiple} and is therefore omitted. 

\begin{Proposition}[Dark front and back solutions] \label{prop:existence_1-front}
Assume that $\sigma_0$ is a simple zero of the Melnikov function $\mathcal{M}$ given by~\eqref{eq:def_Melnikov}. Then, there exist $C,\eta_0>0$ such that for all $\eta \in (0,\eta_0)$ there exists a nondegenerate solution $\underline{\psi}_{\pm,\eta} \colon \R \to \R$ to~\eqref{eq:LLE_defocusing_stationary_real} with 
$$
    \|\underline{\psi}_{\pm,\eta} - \psi_{\pm,\sigma_0}\|_{L^\infty} \leq C \eta, \qquad \underline{\psi}_{\pm,\eta} - v_{\pm,\eta} \in H^2(-1,\infty), \qquad
    \underline{\psi}_{\pm,\eta} - v_{\mp,\eta} \in H^2(-\infty,1),
$$
where $v_{\pm ,\eta} \in H_\per^2(0,T)$ are periodic solutions to~\eqref{eq:LLE_defocusing_stationary_real} with
$$
    \|v_{\pm ,\eta} \mp \sqrt{\zeta} \|_{H_\per^2(0,T)} \leq C \eta.
$$
\end{Proposition}

\begin{Remark}
The fronts $\smash{\underline{\psi}_{\pm,\eta}}$ obtained in Proposition~\ref{prop:existence_1-front} are, in general, degenerate as real-valued solutions to~\eqref{eq:LLE_main_tw_mod_sys} for $\eps=0$, in the sense that the linearization $\smash{\El_0(\underline{\boldsymbol{\psi}}_{\pm,\eta})}$ about $\smash{\underline{\boldsymbol{\psi}}_{\pm,\eta}:=(\underline{\psi}_{\pm,\eta},0)^\top}$ is not invertible. The source of this degeneracy is that the origin belongs to the spectrum of the linearization of~\eqref{eq:LLE_main_tw_mod_sys} about either $\smash{\wb_{+,\eta}:=(v_{+,\eta},0)^\top}$ or $\smash{\wb_{-,\eta}:=(v_{-,\eta},0)^\top}$, and hence $0$ lies in the essential spectrum of $\smash{\El_0(\underline{\boldsymbol{\psi}}_{\pm,\eta})}$; see Figure~\ref{fig:dark_construction}. We overcome this obstruction by concatenating a front and back solution with matching end states, thereby constructing a pulse solution whose asymptotic state is either $v_{-,\eta}$ or $v_{+,\eta}$. The essential spectrum of the corresponding linearization is then determined solely by the spectrum of the linearization of~\eqref{eq:LLE_main_tw_mod_sys} about $\wb_{-,\eta}$ or $\wb_{+,\eta}$, respectively; see Figure~\ref{fig:dark_construction}. We will show that this ultimately allows us to construct nondegenerate pulse solutions to~\eqref{eq:LLE_main_tw_mod_sys}.
\end{Remark}

We show that the front and back solutions from Proposition~\ref{prop:existence_1-front} converge exponentially fast to their periodic end states. We will need this fact for the proof of the nondegeneracy of the dark primary pulse as a solution to~\eqref{eq:LLE_main_tw_mod_sys}. 

\begin{Lemma}[Exponential convergence to periodic end states]\label{lem:exp_convergence_front}
Let $\smash{\underline{\psi}_{\pm,\eta}}$ be a solution to~\eqref{eq:LLE_defocusing_stationary_real}, obtained in Proposition~\ref{prop:existence_1-front}. There exist $\eta_1 \in (0,\eta_0)$ and $\mu >0$ such that for all $\eta \in (0,\eta_1)$ we have
$$
    \lim_{x \to \infty} \left(|\underline{\psi}_{\pm,\eta}(x)-v_{\pm,\eta}(x)|+|\underline{\psi}_{\pm,\eta}'(x)-v_{\pm ,\eta}'(x)|\right) \eu^{\mu x} = 0
$$
and
$$
    \lim_{x \to -\infty} \left(|\underline{\psi}_{\pm,\eta}(x)-v_{\mp,\eta}(x)|+|\underline{\psi}_{\pm,\eta}'(x)-v_{\mp ,\eta}'(x)|\right) \eu^{-\mu x} = 0.
$$
\end{Lemma}
\begin{proof}
We prove the claim only for $\smash{\underline{\psi}_{+,\eta}}$ as the arguments for $\smash{\underline{\psi}_{-,\eta}}$ are identical. It follows from Proposition~\ref{prop:existence_1-front} that we can write
$$
    \underline{\psi}_{+,\eta} = \chi_- (v_{-,\eta} + \sqrt{\zeta}) + \chi_+ (v_{+,\eta} - \sqrt{\zeta}) + \psi_{+,\sigma_0} + w_\eta
$$
where $\chi_\pm \colon \R \to [0,1]$ is a smooth partition of unity satisfying $\text{supp}(\chi_-) = (-\infty,1 ]$ and $\text{supp}(\chi_+) = [1,\infty)$, and where $w_\eta \in H^2(\R)$ is a localized error term. Proposition~\ref{prop:existence_1-front} 
and the continuous embeddings $H^1(\R) \hookrightarrow L^\infty(\R)$ and $H_\per^1(0,T) \hookrightarrow L^\infty(\R)$ 
then yield the bounds
$$
    \|v_{-,\eta} + \sqrt{\zeta}\|_{L^\infty},\|v_{+,\eta} - \sqrt{\zeta}\|_{L^\infty},\|w_\eta\|_{L^\infty},\|w_\eta'\|_{L^\infty} \leq C\eta 
$$
for some $\eta$-independent constant $C>0$. Since $ \psi_{+,\sigma_0}(x) \to \pm \sqrt{\zeta}$ and $ \psi_{+,\sigma_0}'(x) \to 0$ as $x \to \pm \infty$ at an exponential rate, it remains to show that $|w_\eta(x)| + |w'_\eta(x)| \to 0$ exponentially fast as $|x| \to \infty$. To this end, we set $W_\eta(x) := \smash{(w_\eta(x), w_\eta'(x))^\top} \in \R^2$ and observe that $W_\eta$ solves a system of ordinary differential equations
\begin{align}\label{eq:for_W_proof_exp_convergence}
    W_\eta' = \left(A^{(0)} + A_\eta^{(1)}(x) + A_\eta^{(2)}(x)\right)W_\eta + R_\eta(x), \qquad  A^{(0)} =
    \begin{pmatrix}
        0 & 1 \\
        2\zeta & 0 
    \end{pmatrix},
\end{align}
where the coefficient functions $A_{\eta}^{(1)}, A_\eta^{(2)}, R_\eta \in C(\R)$ enjoy the bounds
\begin{align}
\label{eq:A_bounds}
    \|A_\eta^{(1)}(x)\| \leq C \eta, \qquad \|R_\eta(x)\|+\|A_\eta^{(2)}(x)\| \leq C \eu^{-\kappa |x|}, \qquad x \in \R
\end{align}
for some $x$- and $\eta$-independent constants $C,\kappa>0$. 

Since $\zeta>0$, the matrix $A^{(0)}$ is hyperbolic. Hence, the system $W' = \smash{A^{(0)}}W$ admits an exponential dichotomy~\cite{Coppel1978} on $\R$. Roughness of exponential dichotomies against small perturbations~\cite[Proposition~4.1]{Coppel1978} yields that, provided $\eta>0$ is sufficiently small, the system $W' = (A^{(0)} + A^{(1)}_\eta(x)) W$ also admits an exponential dichotomy on $\R$. Denoting by $\Phi_\eta(x,y)$ the evolution of the system $W' = (A^{(0)} + A_\eta^{(1)}(x)) W$, this entails that there exist $\eta$-independent constants $K',\mu'>0$ and projections $P_\eta(x) \in \C^{2 \times 2}$ for $x \in \R$ such that
\begin{align} \label{eq:exp_di}
\begin{split}
    \|\Phi_\eta(x,y) P_\eta(y)\| \leq K' \eu^{-\mu'(x-y)} ,\qquad \text{for }x\geq y,\\
    \|\Phi_\eta(x,y)(I- P_\eta(y))\| \leq K' \eu^{-\mu'(y-x)} ,\qquad \text{for }y\geq x.
    \end{split}
\end{align}
Since $W_\eta$ is a bounded solution of~\eqref{eq:for_W_proof_exp_convergence} and the estimates~\eqref{eq:A_bounds} and~\eqref{eq:exp_di} hold, it must obey the variation-of-constants formula
\begin{align*}
    W_\eta(x) &= \int_{-\infty}^x \Phi_\eta(x,y)P_\eta(y) \Big( A_\eta^{(2)}(y) W_\eta(y) + R_\eta(y) \Big) \de y \\
    & \qquad -\int_x^\infty \Phi_\eta(x,y)(I-P_\eta(y)) \Big( A_\eta^{(2)}(y) W_\eta(y) + R_\eta(y) \Big) \de y, \quad \text{for all }x \in \R.
\end{align*}
Hence, applying the estimates~\eqref{eq:A_bounds} and~\eqref{eq:exp_di} to the latter and taking $\alpha \in (0, \min\{\mu',\kappa\})$, we obtain constants $K,M>0$ such that
$$
    |W_\eta (x)| \eu^{\alpha |x|} \leq K \eu^{\alpha |x|} \int_\R \eu^{-\mu' |x-y|} \eu^{-\kappa|y|} \de y \leq M,
$$
for all $x \in \R$, which completes the proof.
\end{proof}

\paragraph*{Dark primary pulses.} Nondegeneracy of the bifurcating fronts and backs in Proposition~\ref{prop:existence_1-front} as solutions to~\eqref{eq:LLE_defocusing_stationary_real} allows us to apply~\cite[Theorem 3.1]{Bengel2026Multiple} to concatenate a well-separated front and back solution with matching periodic end states to obtain a dark single-pulse solution. This leads to the following result.

\begin{Proposition}[Dark primary pulses] \label{prop:existence_2-front}
Let $\eta_1 > 0$ be as in Lemma~\ref{lem:exp_convergence_front}. Take $\eta \in (0,\eta_1)$ and let $\smash{\underline{\psi}_{\pm , \eta}} \in L^\infty(\R)$ be a nondegenerate front and back solution to~\eqref{eq:LLE_defocusing_stationary_real}, obtained in Proposition~\ref{prop:existence_1-front}. Then, there exist $C,\mu>0$, $N \in \N$, and a sequence $\{a_n\}_{n\geq N} \subset H^2(\R)$ such that $\unu_{\eta,n} =w_{\eta,n}+a_n \in H^2(\R) \oplus H^2_\per(0,T)$ is a nondegenerate solution to~\eqref{eq:LLE_defocusing_stationary_real} for all $n \in \N$ with $n \geq N$, where
$$
    w_{\eta,n}(x) =
    \begin{cases}
        \underline{\psi}_{+,\eta}(x), & x \in \big(-\infty,\frac{1}{2}nT\big], \\
        \underline{\psi}_{-,\eta}(x-nT), & x \in \big(\frac{1}{2}nT,\infty\big).
    \end{cases}
$$
and
\begin{align}\label{eq:bound_error_2-front}
    \|a_n\|_{H^2} \leq C \eu^{-\mu n}.
\end{align}
\end{Proposition}
\begin{proof}
Existence follows from~\cite[Theorem 3.1]{Bengel2026Multiple}. Moreover, the bound~\eqref{eq:bound_error_2-front} is a direct consequence of Lemma~\ref{lem:exp_convergence_front} and the estimate on $a_n$ in~\cite[Theorem 3.1]{Bengel2026Multiple}. Finally, the nondegeneracy of $\unu_{\eta,n}$ as a solution to~\eqref{eq:LLE_defocusing_stationary_real} follows from~\cite[Lemma 3.3]{Bengel2026Multiple}.
\end{proof}

The pulse solution $\unu_{\eta,n}$ from Proposition~\ref{prop:existence_2-front} connects to the end state $v_{-,\eta}$ at $x=\pm\infty$ through the plateau state $v_{+,\eta}$. By the same arguments, one can likewise construct a pulse that connects to $v_{+,\eta}$ at $x=\pm\infty$ through the plateau state $v_{-,\eta}$.

\paragraph*{Nondegeneracy of dark primary pulse to complex-valued formulation.} Let $\unu_{\eta,n}$ be a dark primary pulse solution to~\eqref{eq:LLE_main_tw_mod}, obtained in Proposition~\ref{prop:existence_2-front}, and set $\smash{\ubu_{\eta,n}} = \smash{(\unu_{\eta,n},0)^\top}$. We prove the nondegeneracy of $\ubu_{\eta,n}$ as a solution to system~\eqref{eq:LLE_main_tw_mod_sys} under a nonresonance condition. Proposition~\ref{prop:existence_2-front} yields that $0 \not \in \sigma(L_+(\unu_{\eta,n}))$ and thus we find that $\El_0(\ubu_{\eta,n})$ is invertible if and only if this holds for $L_-(\unu_{\eta,n})$.

To establish invertibility of $L_-(\unu_{\eta,n})$, we adopt a spatial dynamics approach and write the problem $L_-(\unu_{\eta,n})u = 0$ as the first order system
\begin{align}\label{eq:first_order_L-}
    U' = \mathcal{A}_-(\unu_{\eta,n}(x)) U, \qquad
    \mathcal{A}_-(u):=
    \begin{pmatrix}
        0 & 1 \\
        u^2-\zeta & 0 
    \end{pmatrix}
\end{align}
in $U = (u,u')^\top$. Before we start with the analysis of the full equation~\eqref{eq:first_order_L-}, we study the dynamics at the periodic end states $v_{\pm,\eta}$. In particular, we are interested in the spatial Floquet exponents, cf.~\cite[Section~2.1]{KapitulaPromislow2013}, of the $T$-periodic system
\begin{align}\label{eq:first_order_periodic}
    U' = \mathcal{A}_-(v_{\pm,\eta}(x))U,
\end{align}
which we trace for small values of $\eta$. Note that, since $\mathcal{A}_-(v_{\pm,\eta}(x))$ is real-valued with zero trace for all $x \in \R$, we find by Abel's identity, cf.~\cite[Lemma 2.1.31]{KapitulaPromislow2013}, that $\nu$ is a Floquet exponent if and only if $\overline{\nu}$ is, and that the sum of the Floquet exponents $\nu_{1,2}$ satisfies $\nu_1 + \nu_2 = 0 \text{ mod } 2\pi \iu /T$. Therefore, the Floquet exponents of~\eqref{eq:first_order_periodic} are either real or purely imaginary. 

For $\eta=0$, we find 
$$
    \mathcal{A}_-(v_{\pm,0}(x)) =
    \begin{pmatrix}
        0 & 1 \\ 0 & 0
    \end{pmatrix}
$$
by Proposition~\ref{prop:existence_1-front}. Thus, system~\eqref{eq:first_order_periodic} possesses a multiplicity-two Floquet exponent $\nu = 0$ at $\eta = 0$. We prove that this exponent splits as $\eta > 0$ and determine how this splitting depends on the forcing function $F$. For the calculations, it is helpful to characterize imaginary Floquet exponents through eigenvalues of the Bloch operators
\begin{align}\label{def:L-_Bloch}
    L_{-,\xi}(\unu) \colon H_\per^2(0,T) \to L_\per^2(0,T), \qquad L_{-,\xi}(\unu) = \left(\partial_x +\iu\frac{\xi}{T}\right)^2 + \zeta - \unu^2 
\end{align}
with $\xi \in [-\pi,\pi)$. Floquet-Bloch theory~\cite[Section~3.3]{KapitulaPromislow2013} yields that $\nu=\iu \xi$ is a spatial Floquet exponent of~\eqref{eq:first_order_periodic} if and only if the problem
\begin{align}\label{eq:L-_endstates_Bloch}
    L_{-,\xi}(v_{\pm,\eta}) u = 0
\end{align}
has a nontrivial solution $u\in H_\per^2(0,T)$.

\begin{Lemma}[Floquet analysis]\label{lem:splitting_Floquet_exponents}
Assume that $\int_0^T F(x) \de x \neq 0$. Then, there exist $\eta_2 \in (0,\eta_0)$ and $C>0$ such that for all $\eta \in (0,\eta_2)$ equation~\eqref{eq:L-_endstates_Bloch} has a nontrivial solution $u \in H_\per^2(0,T)$ for some $\xi \in \R$ if and only if 
$$v_{\pm,0} \int_0^T F(x) \de x<0.$$ 
In this case, there are precisely two Bloch frequencies $\pm\xi_\eta \in \R$ for which~\eqref{eq:L-_endstates_Bloch} admits a nonzero solution. They satisfy the approximation
\begin{align}\label{eq:expansion_floquet_exponent}
    \left|\xi_\eta^2 + \eta T \frac{v_{\pm,0}}{\zeta} \int_0^T F(x) \de x\right| \leq C \eta^{2}.
\end{align}
In particular, the spatial Floquet exponents of~\eqref{eq:first_order_periodic} are given by $\pm \iu\xi_\eta$. 
\end{Lemma}
\begin{proof}
By standard analytic perturbation theory~\cite{Kato1995} the Floquet exponents of~\eqref{eq:first_order_periodic} are analytic in $\eta^{1/2}$ as they appear from the splitting of a multiplicity-two root. This yields the expansions 
\begin{align}\label{eq:Floquet_expansions}
\begin{split}
    \xi &= \xi_1 \eta^{1/2} + \mathcal{O}(\eta), \qquad
    u = 1 + \eta u_1 + \mathcal{O}(\eta^2),\\
    v_{\pm,\eta} &= \pm \sqrt{\zeta} - \eta (\partial_x^2-2\zeta)^{-1} F + \mathcal{O}(\eta^2).
\end{split}
\end{align}
Inserting~\eqref{eq:Floquet_expansions} into~\eqref{eq:L-_endstates_Bloch}, we arrive at order $\mathcal{O}(\eta)$ at the equation
\begin{align}\label{eq:Floquet_O(1)}
    \partial_x^2 u_1 \pm 2 \sqrt{\zeta}(\partial_x^2-2\zeta)^{-1} F  = \xi_1^2 T^{-2} .
\end{align}
Testing~\eqref{eq:Floquet_O(1)} with $u_0=1$ and integrating by parts, we deduce
\begin{align*}
    \xi_1^2 T^{-1} = \pm 2 \sqrt{\zeta} \langle  F, (\partial_x^2-2\zeta)^{-1}1\rangle_{L_\per^2(0,T)}
    = -\frac{v_{\pm,0}}{\zeta} \int_0^T F(x) \de x.
\end{align*}
We find a real-valued solution $\xi_1$ if and only if $\smash{v_{\pm,0}\int_0^T F(x) \de x <0}$. In this case, we obtain the approximation~\eqref{eq:expansion_floquet_exponent} and the Floquet exponents of~\eqref{eq:first_order_periodic} are purely imaginary for $\eta > 0$ sufficiently small. If $\smash{v_{\pm,0} \int_0^T F(x) \de x >0}$, equation~\eqref{eq:L-_endstates_Bloch} has only the trivial solution. Thus,~\eqref{eq:first_order_periodic} possesses a pair of real Floquet exponents for $\eta > 0$ sufficiently small.
\end{proof}

We are now ready to establish the nondegeneracy of $\ubu_{\eta,n_j}$ for a strictly increasing sequence $\{n_j\}_{j\in \N} \subset \N$ diverging to infinity under a nonresonance condition on $\eta$.

\begin{Lemma}[Nondegeneracy of dark primary pulses]\label{lem:nondeg_L-}
Let $\eta_{1,2} > 0$ be as in Proposition~\ref{prop:existence_2-front} and Lemma~\ref{lem:splitting_Floquet_exponents}, respectively. Let $\smash{\int_0^T F(x) \de x<0}$. Assume that $\eta \in (0,\min\{\eta_1,\eta_2\})$ satisfies the nonresonance condition 
\begin{align}\label{eq:non_resonance}
    \frac{\xi_\eta T}{\pi} \not \in \Q,
\end{align}
where $\xi_\eta>0$ is the positive Bloch frequency from Lemma~\ref{lem:splitting_Floquet_exponents} for $v_{+,\eta}$.
Then, there exists a sequence $\{n_j\}_{j\in \N}\subset\N$ with $n_j \to \infty$ such that the solution $\ubu_{\eta,n_j} = \smash{(\unu_{\eta,n_j},0)^\top}$ to~\eqref{eq:LLE_main_tw_mod_sys}, obtained in Proposition~\ref{prop:existence_2-front}, is nondegenerate for all $j\in \N$, i.e., the operator $\El_0({\ubu_{\eta,n_j}})$ is invertible. 
\end{Lemma}
\begin{proof}
Following the strategy outlined in~\S\ref{sec:method_of_proof}, we first establish exponential dichotomies~\cite{Coppel1978} for system~\eqref{eq:first_order_L-} allowing to reduce~\eqref{eq:first_order_L-} to a boundary-value problem along the plateau state $v_{+,\eta}$ of $\unu_{\eta,n}$.

Since $\smash{v_{-,0} \int_0^TF(x) \de x > 0}$, we deduce from Lemma~\ref{lem:splitting_Floquet_exponents} that system
\begin{align}\label{eq:periodic_system-}
    U'=\mathcal{A}_-(v_{-,\eta}(x)) U
\end{align}
has two real Floquet exponents $\pm\nu_\eta \in\R$ with $\nu_\eta>0$. This implies that~\eqref{eq:periodic_system-} admits an exponential dichotomy on $\R$ with projections of rank one. Using that $\smash{\underline{\psi}_{+,\eta}(x) - v_{-,\eta}(x)} \to 0$ as $x\to-\infty$ by Proposition~\ref{prop:existence_1-front},~\cite[Lemma~3.4]{Palmer1984} yields that  system
\begin{align}
    U' = \mathcal{A}_-(\underline{\psi}_{+,\eta}(x)) U
\end{align}
admits for each $m\in \N$ an exponential dichotomy on $(-\infty,mT]$ with projection $P_{-,m}(x)\in \C^{2\times 2}$ of rank one. Using that $\smash{\underline{\psi}_{-,\eta}(x)-v_{-,\eta}(x)} \to 0$ as $x\to\infty$ we obtain likewise for each $m \in \N$ an exponential dichotomy of system
\begin{align}
    U' = \mathcal{A}_-(\underline{\psi}_{-,\eta}(x)) U
\end{align}
on the interval $[-mT,\infty)$ with projection $P_{+,m}(x) \in \C^{2\times 2}$ of rank one. Fix $m \in \N$. By Proposition~\ref{prop:existence_2-front} we have that
\begin{align*}
    \|\underline{\psi}_{+,\eta}-\unu_{\eta,n}\|_{L^\infty(-\infty,mT)} \to 0,\qquad
    \|\underline{\psi}_{-,\eta}(\cdot-nT)-\unu_{\eta,n}\|_{L^\infty((n-m)T,\infty)} \to 0, \qquad
    \text{as } n \to \infty.
\end{align*}
Thus, provided $n \in \N$ is sufficiently large, roughness of exponential dichotomies~\cite[Proposition~4.1]{Coppel1978} implies that system~\eqref{eq:first_order_L-} admits exponential dichotomies with $n$-independent constants on the intervals $(-\infty,mT]$ and $[(n-m)T,\infty)$ with projections $\smash{\tilde{P}_{-,m,n}(x)} \in \C^{2\times 2}$ and $\smash{\tilde{P}_{+,m,n}(x)} \in \C^{2\times 2}$, respectively, both of rank one and, moreover, we have
\begin{align}\label{eq:approximation_projections}
\|\tilde{P}_{-,m,n}-P_{-,m}\|_{L^\infty(-\infty,mT)} \to 0, \qquad
\|\tilde{P}_{+,m,n}-P_{+,m}(\cdot-nT)\|_{L^\infty((n-m)T,\infty)} \to 0
\end{align}
as $n \to \infty$.

We use the exponential dichotomies to reduce~\eqref{eq:first_order_L-} to a boundary-value problem on a finite interval. To this end, let $q_{-,m}\in \R^2\setminus\{0\}$ be in the range of $P_{-,m}(mT)$. Note, that we can choose $q_{-,m}$ real-valued since $\A_-(\unu_{\eta,n}(x))$ is real-valued and $\text{dim}(\text{ran}(P_{-,m}(mT))) = 1$. We set $p_{-,m}:=q_{-,m}/\|q_{-,m}\|$. Similarly, we choose $q_{+,m} \in \R^2\setminus\{0\}$ in the range of $P_{+,m}(-mT)$ and we set $p_{+,m}:=q_{+,m}/\|q_{+,m}\|$. Now we define $\tilde{p}_{-,m,n}:=\tilde{P}_{-,m,n}(mT)p_{-,m}$ and $\tilde{p}_{+,m,n}:=\tilde{P}_{+,m,n}((n-m)T)p_{+,m}$ and observe by~\eqref{eq:approximation_projections} that $\|\tilde{p}_{\pm,m,n}-p_{\pm,m}\|\to 0$ as $n \to \infty$. In particular, we find $\|\tilde{p}_{\pm,m,n}\|\to 1$ as $n \to \infty$. Thus, the problem~\eqref{eq:first_order_L-} has a nonzero $L^2$-localized solution for $n\in \N$ sufficiently large with $2m<n$ if and only if the boundary-value problem
\begin{align}\label{eq:nondegeneracy_bvp}
\begin{split}
    U' &= \mathcal{A}_-(\unu_{\eta,n}(x))U, \qquad x \in [mT,(n-m)T],\\
    U(mT) &\in \text{span}\{\tilde{p}_{-,m,n}\}, \qquad
    U((n-m)T) \in \text{span}\{\tilde{p}_{+,m,n}\},
\end{split}
\end{align}
has a nonzero solution.

We analyze the boundary-value problem~\eqref{eq:nondegeneracy_bvp}. To this end, we take $n\in \N$ with $2m<n$ and note that for all $x \in [mT,(n-m)T]$ we have the estimate
\begin{align*}
    &\|\mathcal{A}_-(\unu_{\eta,n}(x))-\mathcal{A}_-(v_{+,\eta}(x))\| \\
    &\qquad\qquad \leq \|\mathcal{A}_-(\unu_{\eta,n}(x))-\mathcal{A}_-(w_{\eta,n}(x))\| +\|\mathcal{A}_-(w_{\eta,n}(x))-\mathcal{A}_-(v_{+,\eta}(x))\| \\
    &\qquad\qquad \leq C_1 \left( \eu^{-\mu n} + \eu^{-\mu x} + \eu^{\mu (x-nT)} \right)
\end{align*}
where $C_1>0$ is an $m$- and $n$-independent constant and $\mu>0$ is the minimum of the two exponential decay rates from Lemma~\ref{lem:exp_convergence_front} and Proposition~\ref{prop:existence_2-front}. We use this estimate to relate the dynamics of $U' = \mathcal{A}_-(\unu_{\eta,n}(x))U$ to the dynamics of the $T$-periodic system 
\begin{align}\label{eq:first_order_v_+}
    U' = \mathcal{A}_-(v_{+,\eta}(x))U.   
\end{align}
Let $\mathcal{T}(x,y)$ denote the evolution of~\eqref{eq:first_order_v_+}. Using Floquet's  theorem~\cite[Theorem 2.14]{Meiss2007Differential}, Lemma~\ref{lem:splitting_Floquet_exponents}, and the fact that $\smash{v_{+,0} \int_0^T F(x) \de x = -v_{-,0} \int_0^T F(x) \de x < 0}$, we obtain a family of invertible matrices $Q(x)\in \R^{2\times 2}$, which are continuously differentiable and $2T$-periodic in $x$, such that
$\mathcal{T}(x,y) = Q(x) \eu^{B(x-y)}Q(y)^{-1}$ with
$$
    \eu^{B(x-y)} =
    \begin{pmatrix}
        \cos(\xi_\eta (x-y)) & -\sin(\xi_\eta (x-y)) \\ \sin(\xi_\eta (x-y)) & \cos(\xi_\eta (x-y))
    \end{pmatrix}, \qquad x,y \in \R.
$$
Let $U$ be the solution to the initial-value problem
\begin{align}\label{eq:initial_value_problem_plateau}
    U' &= \mathcal{A}_-(\unu_{\eta,n}(x))U, \qquad U(mT)=\tilde{p}_{-,m,n}.
\end{align}
Then, the boundary-value problem~\eqref{eq:nondegeneracy_bvp} has a nontrivial solution if and only if $\det(U((n-m)T)|\tilde{p}_{+,m,n}) = 0$. 

Let $R_{\eta,n}(x):=\mathcal{A}_-(\unu_{\eta,n}(x))-\mathcal{A}_-(v_{+,\eta}(x))$. We use the variation-of-constants formula to write the solution to~\eqref{eq:initial_value_problem_plateau} as
\begin{align}\label{eq:Duahmel}
    U(x) = \mathcal{T}(x,mT)\tilde{p}_{-,m,n} + \int_{mT}^x \mathcal{T}(x,y) R_{\eta,n}(y)U(y) \de y.
\end{align}
We bound~\eqref{eq:Duahmel} on the interval $[mT,(n-m)T]$ in order to relate it to the solution of the periodic system~\eqref{eq:first_order_v_+}. Clearly, there exists $C_2>0$ such that
$$
    \|\mathcal{T}(x,y)\|\leq C_2, \qquad \text{for all } x,y\in \R.
$$
Thus, using the bounds for $\mathcal{T}(x,y)$ and $R_{\eta,n}(y)$, we find
\begin{align*}
    \|U(x)\| &\leq C_2 \left(\|\tilde{p}_{-,m,n}\| + \int_{mT}^x \|R_{\eta,n}(y)\| \de y \max_{y\in [mT,x]}\|U(y)\| \right) \\
    &\leq C_3 \left(1 +
    \int_{mT}^{(n-m)T} \left(\eu^{-\mu n} + \eu^{-\mu y} + \eu^{\mu (y-nT)} \right) \de y  \max_{y\in [mT,(n-m)T]}\|U(y)\| \right) \\
    & \leq C_4 \left( 1 + (\eu^{-\mu n/2} + \eu^{-\mu mT})\max_{y\in [mT,(n-m)T]}\|U(y)\| \right)
\end{align*}
for all $x \in [mT,(n-m)T]$ and some $x$-, $m$-, and $n$-independent constants $C_{2,3,4}>0$. Taking the maximum over all $x \in [mT,(n-m)T]$ and taking $n,m \in \N$ so large that $C_4( \eu^{-\mu n/2} + \eu^{-\mu mT})<1/2$, we obtain
$$
    \max_{y\in [mT,(n-m)T]}\|U(y)\| \leq 2C_4.
$$
This allows us to estimate
\begin{align}\label{eq:approximation_on_periodic_evolution_1}
\begin{split}
    \|U(x) - \mathcal{T}(x,mT)\tilde{p}_{-,m,n}\|
    &\leq \int_{mT}^{(n-m)T} \|\mathcal{T}(x,y) R_{\eta,n}(y)U(y)\| \de y \\
    &\leq 2 C_2 C_4 \int_{mT}^{(n-m)T} \|R_{\eta,n}(y)\| \de y
    \leq C_5 (\eu^{-\mu n/2} + \eu^{-\mu mT})
\end{split}
\end{align}
for some $m$- and $n$-independent constant $C_5>0$ and all $x \in [mT,(n-m)T]$. Thus,~\eqref{eq:approximation_on_periodic_evolution_1} yields
\begin{align}\label{eq:approximation_on_periodic_evolution_2}
    \|U(x) - \mathcal{T}(x,mT)p_{-,m}\| \leq   C_5 (\eu^{-\mu n/2} + \eu^{-\mu mT}) + C_2\|\tilde{p}_{-,m,n} - p_{-,m}\|.
\end{align}
We set $Q_0:=Q(0)$ and assume that $n,m$ are even. Then, using that $Q$ is $2T$-periodic together with the formula for $\mathcal{T}$, we infer from~\eqref{eq:approximation_on_periodic_evolution_2} the bound
\begin{align}\label{eq:approximation_on_periodic_evolution_3}
    \|U((n-m)T) - Q_0^{-1} \eu^{B(n-2m)T}Q_0p_{-,m}\| \leq   C_5 (\eu^{-\mu n/2} + \eu^{-\mu mT}) + C_2\|\tilde{p}_{-,m,n} - p_{-,m}\|.
\end{align}
For $\theta\in [0,1]$ we set
$$
    R_\theta:=
    \begin{pmatrix}
        \cos(2\pi\theta) & -\sin(2\pi\theta) \\ \sin(2\pi\theta) & \cos(2\pi\theta)
    \end{pmatrix}.
$$
Then, there exists an $m$- and $n$-independent constant $C_6>0$ with
\begin{align}\label{eq:estimate_determinant}
\begin{split}
    &|\det(U((n-m)T)|\tilde{p}_{+,m,n})-\det(Q_0^{-1} R_\theta Q_0p_{-,m}|p_{+,m})|\\
    &\quad\leq C_6 \left( \eu^{-\mu n/2} + \eu^{-\mu mT}+\|\tilde{p}_{-,m,n} - p_{-,m}\|+\|\tilde{p}_{+,m,n} - p_{+,m}\|+ \|R_\theta-\eu^{B(n-m)T}\| \right).
\end{split}
\end{align}
Since $Q_0p_{\pm,m}\in \R^2$ with $\|Q_0p_{\pm,m}\|\geq 1/\|Q_0^{-1}\|$, there exists for every $m\in\N$ an angle $\theta\in [0,1]$ such that
\begin{align}\label{eq:estimate_on_Evansfunction_0}
    |\det(R_\theta Q_0p_{-,m}|Q_0p_{+,m})|\geq \frac{1}{\|Q_0^{-1}\|^2},
\end{align}
where we use that $|\det(v_1|v_2)|$ is the area of the parallelogram spanned by $v_1,v_2 \in \R^2$. This implies that
$$
    |\det(Q_0^{-1} R_\theta Q_0p_{-,m}|p_{+,m})| = |\det(Q_0^{-1})||\det(R_\theta Q_0p_{-,m}|Q_0p_{+,m})|
    \geq \frac{|\det(Q_0^{-1})|}{\|Q_0^{-1}\|^2}=:A.
$$
Subsequently, we fix $m_0\in \N$ even such that $C_6 \eu^{-\mu m_0T} < A/4$ and let $\theta_0 \in [0,1]$ such that~\eqref{eq:estimate_on_Evansfunction_0} holds. Since $\eta$ satisfies the nonresonance condition~\eqref{eq:non_resonance}, we infer by~\cite{Kronecker} the existence of a strictly increasing sequence $\{\tilde{n}_j\}_{j\in \N}\subset 2\N$ with $\tilde{n}_{j} \uparrow \infty$ such that $\tilde{n}_j\xi_\eta T/(2\pi) \mod 1 \to \theta_0$ as $j \to \infty$. We set $n_j:=\tilde{n}_j+m_0$. Then, we choose $j_0\in \N$ so large that
$$
    C_6 \left( \eu^{-\mu n_{j}/2} +\|\tilde{p}_{-,m_0,n_{j}} - p_{-,m_0}\|+\|\tilde{p}_{+,m_0,n_{j}} - p_{+,m_0}\|+ \|R_{\theta_0}-\eu^{B\tilde{n}_{j}T}\| \right)\leq \frac{A}{4}
$$
for all $j \geq j_0$. Estimate~\eqref{eq:estimate_determinant} implies that for all $j\geq j_0$ we have
$$
    |\det(U((n_j-m_0)T)|\tilde{p}_{+,m_0,n_j})-\det(Q_0^{-1} R_{\theta_0} Q_0p_{-,m_0}|p_{+,m_0})|\leq \frac{A}{2}.
$$
But since $|\det(Q_0^{-1} R_{\theta_0} Q_0p_{-,m_0}|p_{+,m_0})|\geq A > 0$, we deduce that $\det(U((n_j-m_0)T)|\tilde{p}_{+,m_0,n_j})\neq 0$ for all $j\geq j_0$. This means that for $m=m_0$ and $n = n_j$ with $j \geq j_0$ the boundary-value problem~\eqref{eq:nondegeneracy_bvp} has only the zero solution. So, the linear operator $L_-(\unu_{\eta,n_j})$ is injective. Since,~\eqref{eq:first_order_L-} has exponential dichotomies on $(-\infty,0]$ and $[0,\infty)$ with projections of rank one, we find by~\cite[Lemma~4.2]{Palmer1984} that the operator $L_-(\unu_{\eta,n_j})$ is Fredholm of index zero. In particular, injectivity of $L_-(\unu_{\eta,n_j})$ implies its invertibility. Thus, the operator $\El_0(\ubu_{\eta,n_j})$ is invertible and we have established nondegeneracy of $\ubu_{\eta,n_j}$ for all $j\in \N$.
\end{proof}

In the final step, we combine Proposition~\ref{prop:existence_2-front} with Lemma~\ref{lem:nondeg_L-} to construct nondegenerate stationary dark single-pulse solutions to~\eqref{eq:LLE_main_tw_mod_sys} for $0 < \eps \ll \eta \ll 1$.

\begin{Proposition}[Dark primary pulses for $\eps > 0$]\label{prop:existence_periodic_2-front}
Let $\eta_{1,2} > 0$ be as in Proposition~\ref{prop:existence_2-front} and Lemma~\ref{lem:splitting_Floquet_exponents}, respectively. Let $\smash{\int_0^T F(x) \de x<0}$. Assume that $\eta \in (0,\min\{\eta_1,\eta_2\})$ satisfies the nonresonance condition~\eqref{eq:non_resonance}. Let $\ubu_{\eta,n_j}$ be a nondegenerate solution to~\eqref{eq:LLE_main_tw_mod_sys}, obtained in Proposition~\ref{prop:existence_2-front}, with $\{n_j\}_{j\in \N}$ as in Lemma~\ref{lem:nondeg_L-}. Fix $j \in \N$.

Then, there exists $\eps_{0}>0$ such that for each $\eps \in (0,\eps_{0})$ there exists a nondegenerate stationary solution $\ubu_{\eta,n_j,\eps} \in H^2(\R) \oplus H_\per^2(0,T)$ to~\eqref{eq:LLE_main_tw_mod_sys} with
\begin{align*}
\lim_{\eps \downarrow 0} \|\ubu_{\eta,n_j,\eps} - \ubu_{\eta,n_j}\|_{L^\infty} = 0.
\end{align*}
\end{Proposition}
\begin{proof}
Set $X:=H^2(\R)\oplus H^2_\per(0,T)$ and $Y:=L^2(\R)\oplus L^2_\per(0,T)$. We introduce the smooth nonlinear map $\mathcal{F}_\eta\colon X\times\R\to Y$, given by
$$
    \mathcal{F}_\eta(\ub,\eps)
    =
    J\big(\ub_{xx}+\zeta\ub-|\ub|^2\ub\big)
    +\eps(\omega\ub_x-\ub)+\eta J\mathbf F.
$$
Then, it holds $\smash{\mathcal{F}_\eta(\ubu_{\eta,n_j},0)=0}$ and $\smash{\partial_{\ub}\mathcal{F}_\eta(\ubu_{\eta,n_j},0) = \El_0(\ubu_{\eta,n_j})}$, where $\El_0(\ubu_{\eta,n_j})$ is the linearization operator given by~\eqref{eq:defL_eps}, now acting from $X$ into $Y$. We claim that $\partial_{\ub}\mathcal{F}_\eta(\ubu_{\eta,n_j},0)$ is an isomorphism. To prove this, set $\smash{\vb_{-,\eta}:=(v_{-,\eta},0)^\top}$. Then, we have $\smash{\ubu_{\eta,n_j}-\vb_{-,\eta}\in H^2(\R)}$ by Propositions~\ref{prop:existence_1-front} and~\ref{prop:existence_2-front} and, thus, the difference $\El_0(\ubu_{\eta,n_j})- \El_0(\vb_{-,\eta})$ is given by multiplication by a function in $H^2(\R)$. Hence, with respect to the decompositions $X = H^2(\R) \oplus H^2_\per(0,T)$ and $Y = L^2(\R) \oplus L^2_\per(0,T)$, this gives the upper-triangular block representation 
$$
    \El_0(\ubu_{\eta,n_j})=
    \begin{pmatrix}
    \mathfrak{L}_{\mathrm{loc}} & \mathfrak{K}\\
    0 & \mathfrak{L}_\per
    \end{pmatrix},
$$
where $\smash{\mathfrak{L}_{\mathrm{loc}} = \El_0(\ubu_{\eta,n_j})|_{H^2(\R)}}$, $\smash{\mathfrak{L}_\per = \El_0(\vb_{-,\eta})|_{H_\per^2(0,T)}}$, and $\mathfrak{K} := \smash{\El_0(\ubu_{\eta,n_j})- \El_0(\vb_{-,\eta})}$ acts as an operator from $H_\per^2(0,T)$ into $L^2(\R)$. The operators $\mathfrak{L}_\mathrm{loc}$ and $\mathfrak{L}_\per$ on the diagonal are invertible by Lemmas~\ref{lem:nondeg_L-} and~\ref{lem:splitting_Floquet_exponents}, respectively. Therefore, $\El_0(\ubu_{\eta,n_j})\colon X\to Y$ is also invertible, as claimed. The implicit function theorem now yields $\eps_{0}>0$ and a smooth branch $(-\eps_{0},\eps_{0})\to X$, $\eps\mapsto\ubu_{\eta,n_j,\eps}$ of stationary solutions with $\ubu_{\eta,n_j,\eps}\to\ubu_{\eta,n_j}$ in $X$ as $\eps\to0$. In particular, the $L^\infty$-convergence follows from the continuous embedding $X \hookrightarrow L^\infty(\R)$. Finally, the nondegeneracy of $\ubu_{\eta,n_j,\eps}$ is a direct consequence of the smoothness in $\eps$.
\end{proof}

\paragraph*{Proofs of Theorem~\ref{thm:main_dark} and Corollary~\ref{cor:dark_FC}.} The proof of Theorem~\ref{thm:main_dark} follows directly from Propositions~\ref{prop:existence_1-front},~\ref{prop:existence_2-front}, and~\ref{prop:existence_periodic_2-front}. The dark primary pulses constructed in Proposition~\ref{prop:existence_periodic_2-front} are nondegenerate. Hence, \cite[Theorem~3.1]{Bengel2026Multiple} yields stationary $M$-pulse solutions to~\eqref{eq:LLE_main_tw_mod_sys}, which arise near the formal concatenation of $M$ primary pulses. By \cite[Lemma~3.3]{Bengel2026Multiple}, these $M$-pulse solutions are nondegenerate. Finally, \cite[Theorem~4.1]{Bengel2026Multiple} allows us to periodically extend the $M$-pulse solutions, yielding the soliton-based frequency-comb solutions stated in Corollary~\ref{cor:dark_FC}, thereby completing the proof.

\section{Numerical simulations}\label{sec:numerics}
In this section, we present numerical simulations corroborating the existence and stability results, obtained in Theorems~\ref{thm:main_bright} and~\ref{thm:main_dark} and Corollaries~\ref{cor:bright_FC} and~\ref{cor:dark_FC}. All approximations of stationary solutions to~\eqref{eq:LLE2} are computed with the Matlab continuation package \texttt{pde2path}~\cite{pde2path}. We impose periodic boundary conditions in the spatial variable $x$ throughout.

\begin{figure}[b!]
 \centering
 \includegraphics[width=\textwidth]{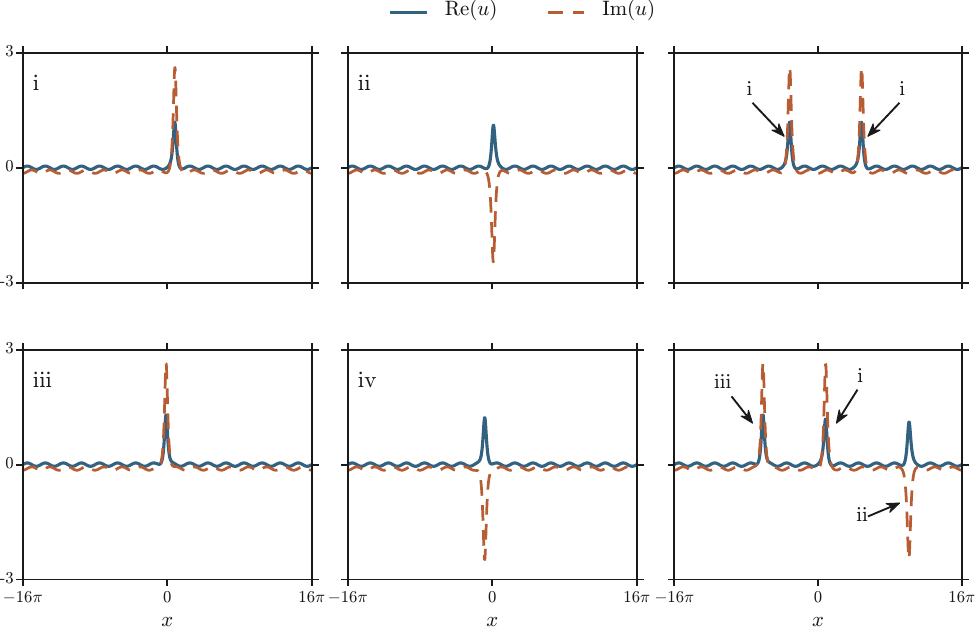}
 \caption{Bright single- and multipulse solutions to~\eqref{eq:LLE_main} for $d=1$, $\zeta=4$, $\omega=1$, $f_0=4$, $f_1=2.1$, and $\eps=0.1$, computed with periodic boundary conditions. The four primary pulses \textup{i}--\textup{iv} correspond to the four roots of the Melnikov conditions in~\eqref{eq:numerics_melnikov_roots}. The stability has been checked computing the eigenvalues of the discretized linearized operator and yields that solution~\textup{i} is spectrally stable, solutions~\textup{ii} and~\textup{iii} are unstable through real eigenvalues, and solution~\textup{iv} exhibits an oscillatory instability. The right panels show a stable two-pulse composed of two copies of~\textup{i} and an unstable three-pulse composed of \textup{iii}, \textup{i}, and~\textup{ii}.}
 \label{fig:bright_profiles}
\end{figure}

Let us first consider the anomalous dispersion regime. We then focus on stationary solutions to equation~\eqref{eq:LLE_main}. In accordance with the physically relevant setting of bichromatic forcing, we set $F(x)=f_0+f_1\mathrm{e}^{\mathrm{i}x}$; see Remark~\ref{rem:derivation}. For this choice of forcing, the Melnikov condition in Theorem~\ref{thm:main_bright} can be evaluated explicitly and yields, modulo $2\pi$, the four solutions
\begin{equation}\label{eq:numerics_melnikov_roots}
 \theta_0=\pm\theta_*,\qquad \sigma_0=\theta_0\pm\frac{\pi}{2},\qquad
 \theta_*:=\arccos\left(\frac{4\sqrt{\zeta}}
 {\pi\sqrt{2}\,f_0}\right).
\end{equation}
\begin{figure}[b!]
 \centering
 \includegraphics[width=.9\textwidth]{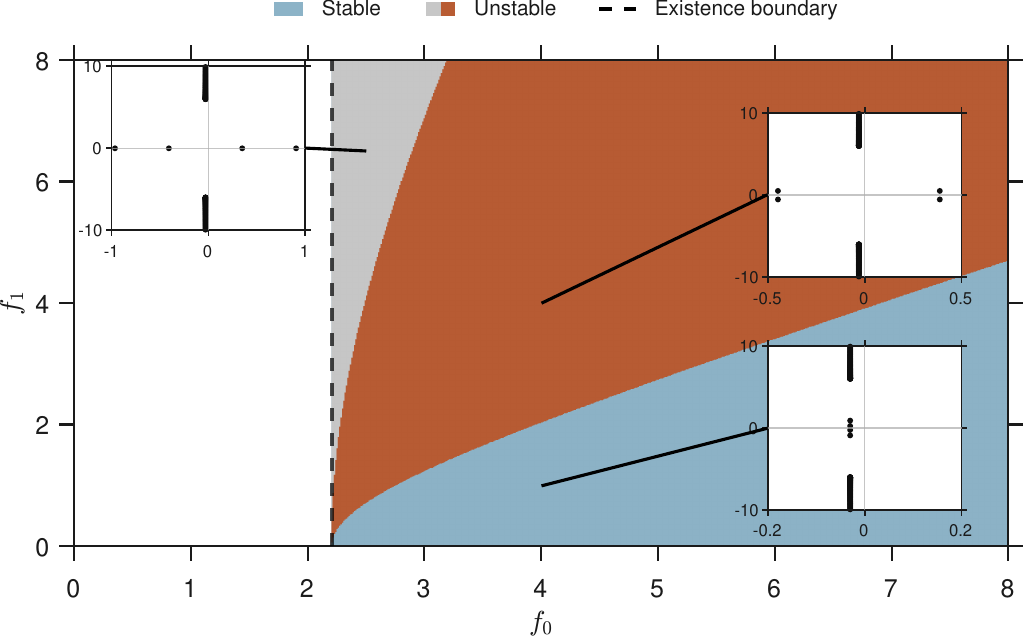}
 \caption{Stability chart in the $(f_0,f_1)$-plane for $d=1$, $\zeta=6$, and $\omega=1$, on the branch selected by \eqref{eq:numerics_selected_root}. The dashed line marks the existence boundary. The blue region corresponds to four distinct purely imaginary roots of the reduced Evans function, the gray region to a real instability, and the copper region to an oscillatory instability. The stability information is obtained from the reduced Evans function \eqref{eq:numerics_reduced_matrix} and confirmed by the insets that show the eigenvalues of the discretized linearization at $\eps=0.1$ for the indicated parameter values.}
 \label{fig:stability_chart}
\end{figure}
In particular, these roots exist provided that the existence condition $f_0^2\pi^2>8\zeta$ holds. Strikingly, they do not dependent on $f_1$ and thus the condition coincides with the one derived in the case of monochromatic forcing ($f_1 =0$); see, for instance,~\cite{Bengel2025Existence}. It is also useful to record the corresponding explicit form of the matrices entering the reduced Evans function $D_{\theta_0,\sigma_0}$. Setting
$$
 \alpha=f_1\pi\sqrt{2}\,
 \operatorname{sech}\!\left(\frac{\pi}{2\sqrt{\zeta}}\right)
 \sin(\sigma_0-\theta_0),
 \qquad
 \beta=f_0\pi\sqrt{2}\sin(\theta_0),
$$
we find
\begin{equation}\label{eq:numerics_reduced_matrix}
 A_{\theta_0,\sigma_0}=
 \begin{pmatrix}
  \alpha & \alpha\\
  \alpha & \alpha-\beta
 \end{pmatrix},
 \qquad
 B=
 \begin{pmatrix}
  \sqrt{\zeta}&0\\
  0&-1/\sqrt{\zeta}
 \end{pmatrix},
 \qquad
 D_{\theta_0,\sigma_0}(\lambda)
 =\det\!\left(A_{\theta_0,\sigma_0}+\lambda^2B\right).
\end{equation}

For the computations in Figure~\ref{fig:bright_profiles}, we chose the parameters $d=1, \zeta=4, \omega=1,f_0=4, f_1=2.1,$ and continue each of the four primary pulses selected by~\eqref{eq:numerics_melnikov_roots} from the NLS limit to $\eps=0.1$. The bifurcating solutions are labeled \textup{i}--\textup{iv} in the left and middle panels of Figure~\ref{fig:bright_profiles} and all four profiles consist of a strongly localized pulse superposed on a small-amplitude periodic background, in agreement with the approximation established in Theorem~\ref{thm:main_bright}. We determine their numerical stability by computing the eigenvalues of the discretized linearization. While solution~\textup{i} is spectrally stable, solutions~\textup{ii} and~\textup{iii} possess real unstable eigenvalues and  solution~\textup{iv} exhibits an oscillatory instability caused by a nonreal complex-conjugate pair in the open right half-plane. The right panels of Figure~\ref{fig:bright_profiles} illustrate well-separated multipulses: the two-pulse formed from two copies of the stable primary pulse~\textup{i} is spectrally stable, the three-pulse assembled from the primary pulses \textup{iii}, \textup{i}, and~\textup{ii} is unstable, in agreement with the statements in Corollary~\ref{cor:bright_FC}.

\begin{figure}[b!]
 \centering
 \includegraphics[width=.75\textwidth]{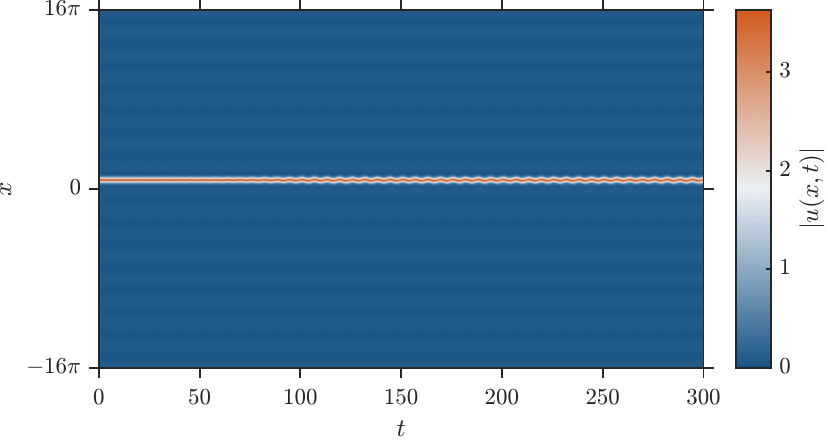}
 \caption{Time evolution of the modulus $|u(x,t)|$ of a perturbed primary bright pulse for
 $d=1$, $\zeta=6$, $\omega=1$, $f_0=4$, $f_1=2.15$, and $\eps=0.1$. The initial condition is a small perturbation of the associated oscillatory unstable stationary pulse on the branch selected by \eqref{eq:numerics_selected_root}.}
 \label{fig:oscillatory_dynamics}
\end{figure}

We next use the reduced Evans function~\eqref{eq:numerics_reduced_matrix} to examine how the stability of the primary pulse depends on the forcing amplitudes $f_0,f_1$. To this end, we set $d=1$, $\zeta=6$, and $\omega=1$, vary $(f_0,f_1)$, and select the Melnikov root
\begin{equation}\label{eq:numerics_selected_root}
 \theta_0=\arccos\!\left(\frac{4\sqrt{\zeta}}
 {\pi\sqrt{2}\,f_0}\right),
 \qquad
 \sin(\sigma_0-\theta_0)=1.
\end{equation}
The stability chart in Figure~\ref{fig:stability_chart} is obtained by computing the eigenvalues of the $2\times2$ matrix pencil in \eqref{eq:numerics_reduced_matrix}, or, equivalently, the roots of $D_{\theta_0,\sigma_0}$. In the stable region, $D_{\theta_0,\sigma_0}$ has four distinct purely imaginary roots which leads to all spectrum residing on the the line $\Re(z) = -0.1$. The two unstable regions are distinguished by whether the reduced problem possesses a positive real root (instability through a real unstable eigenvalue) or a genuinely complex root with positive real part (instability through a pair of conjugate unstable eigenvalues). These correspond, respectively, to a nonoscillatory and an oscillatory instability. We compute the spectrum of the discretized full linearization at $\eps=0.1$ for the three parameter values indicated by the insets in Figure~\ref{fig:stability_chart}. Finally, the dashed vertical line represents the existence boundary $f_0\pi\sqrt{2}=4\sqrt{\zeta}$.

To illustrate the destabilization mechanism generated by the oscillatory instability, we further compute a primary pulse for the parameter values as in Figure~\ref{fig:stability_chart} with $(f_0,f_1)=(4,2.15)$. These parameter values lie just above the oscillatory stability boundary in Figure~\ref{fig:stability_chart}. We then take a small perturbation of the resulting unstable stationary pulse as initial datum for a direct time integration of~\eqref{eq:LLE2}. As shown in Figure~\ref{fig:oscillatory_dynamics}, the perturbed pulse exhibits small oscillations that start growing and lead to a temporally oscillating pulse. The dynamics are consistent with a Hopf bifurcation and are to leading order governed by the amplitude equation for the two critical oscillatory modes~\cite{ChangPromislow2007}.

\begin{figure}[t]
 \centering
 \includegraphics[width=\textwidth]{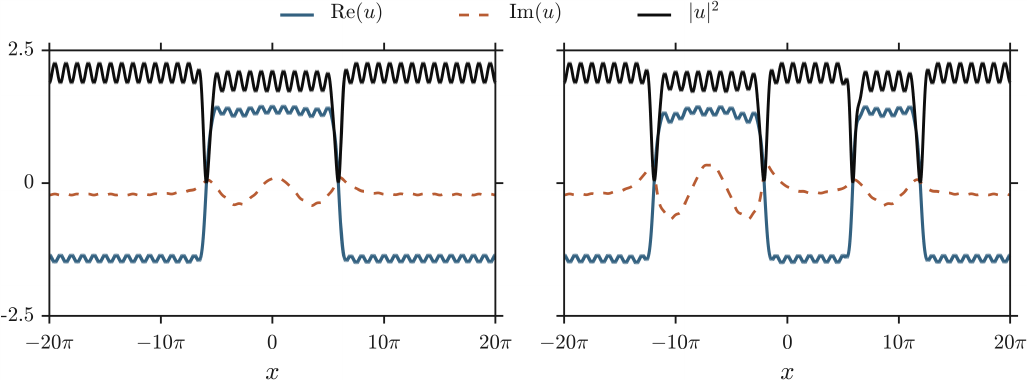}
 \caption{Dark pulse solutions to~\eqref{eq:LLE_main_tw_mod} for $d=-1$, $\zeta=2$, $\eps=0.01$, $\eta=0.55$, $\omega=0.5$, and $F(x)=-0.2+\cos(2x)$. The left panel shows a dark primary pulse formed by a front and a back, while the right panel shows a dark two-pulse containing four domain walls.}
 \label{fig:dark_profiles}
\end{figure}

Finally, we consider the case of normal dispersion and focus on stationary solutions to~\eqref{eq:LLE_main_tw_mod}. We compute approximations of the dark pulse solutions from Theorem~\ref{thm:main_dark} and Corollary~\ref{cor:dark_FC}, taking $d=-1,\zeta=2, \omega=0.5, F(x)=-0.2+\cos(2x)$. We perform numerical path continuation from the NLS limit up to $\eps=0.01$ and $\eta=0.55$. The left panel of Figure~\ref{fig:dark_profiles} shows a dark primary pulse. Its real part consists of a front and a back connecting to the same small-amplitude periodic end state through a long periodic plateau near the opposite constant NLS state. Correspondingly, the intensity $|u|^2$ has two pronounced minima, one at each domain wall. The right panel shows a dark two-pulse near the formal concatenation of two such primary pulses; see Corollary~\ref{cor:dark_FC}. It contains four well-separated domain walls and two intermediate plateau states.

\appendix
\section{Technical computations}\label{app:computations}
We present the proof of Lemma~\ref{lem:formula_for_matrices_AB}.

\begin{proof}[Proof of Lemma~\ref{lem:formula_for_matrices_AB}]
We begin by computing the entries of the matrix $\mathcal{A}$, defined in~\eqref{eq:matrices_reduces_evp}. To facilitate the calculations, we introduce the matrix
$$
    T = \frac{1}{2}
    \begin{pmatrix}
        1 & 1 \\ -\iu & \iu
    \end{pmatrix} \in \C^{2\times 2} \qquad \text{with inverse}\qquad
    T^{-1} =
    \begin{pmatrix}
        1 & \iu \\ 1 & -\iu
    \end{pmatrix},
$$
which represents the change of coordinates from $(u,\overline{u})^\top$ to $(\Re(u),\Im(u))^\top$. In particular, we have $\tilde{\mathbf{p}}_0 := T^{-1} \mathbf{p}_0 = ( u_0',\overline{u_0'})^\top$, $\tilde{\mathbf{q}}_0 := T^{-1} \mathbf{q}_0 = (\iu u_0,-\iu\overline{u_0})^\top$, and
\begin{align*}
    \tilde{L}_1 := T^* L_1 T =
    \begin{pmatrix}
        -u_0 \overline{u_1} - \overline{u_0} u_1 & - u_0 u_1 \\
        - \overline{u_0 u_1} & - u_0 \overline{u_1} - \overline{u_0} u_1
    \end{pmatrix} 
    + \frac{\omega}{2} 
    \begin{pmatrix}
        \iu & 0 \\ 0 & -\iu
    \end{pmatrix} \partial_x.
\end{align*}
Consequently, we obtain
$$
    \mathcal{A} = 
    \begin{pmatrix}
        \langle \tilde{L}_1 \tilde{\mathbf{p}}_0,\tilde{\mathbf{p}}_0\rangle_{L^2} &
        \langle \tilde{L}_1 \tilde{\mathbf{p}}_0,\tilde{\mathbf{q}}_0\rangle_{L^2} \\
        \langle \tilde{L}_1 \tilde{\mathbf{q}}_0,\tilde{\mathbf{p}}_0\rangle_{L^2} &
        \langle \tilde{L}_1 \tilde{\mathbf{q}}_0,\tilde{\mathbf{q}}_0\rangle_{L^2} 
    \end{pmatrix}.
$$
We proceed with computing the scalar products. We have
\begin{align}
    \tilde{L}_1 \tilde{\mathbf{p}}_0 &=\frac{\omega}{2}  
    \begin{pmatrix}
        \iu u_0'' \\ -\iu \overline{u_0''}
    \end{pmatrix}
    -
	\begin{pmatrix}
		u_0 \overline{u_1} u_0' + \overline{u_0} u_1 u_0' + u_0 u_1 \overline{u_0'} \\
		\overline{u_0}u_1 \overline{u_0'} + u_0 \overline{u_1 u_0'} + \overline{u_0 u_1} u_0'
	\end{pmatrix}.
\end{align}
So, combining this with the formula~\eqref{eq:expansion_u} for $u_1$, we deduce
\begin{align*}
	\langle \tilde{L}_1 \tilde{\mathbf{p}}_0, \tilde{\mathbf{p}}_0 \rangle_{L^2}
	&= -4 \Re \int_\R u_0 \overline{u_1} |u_0'|^2\de x - 2\Re \int_\R u_0 u_1 (\overline{u_0'})^2\de x \\
	&= -6 \Re \int_\R |u_0'|^2 u_0 \overline{(v_1 + w_1 )} \de x,
\end{align*}
where we used that $u_0 \overline{u_0'} = \overline{u_0} u_0'$, $u_0 \overline{u_0''} = \overline{u_0} u_0''$ and that
\begin{align*}
	\int_\R |u_0'|^2 u_0 \overline{u_0'} \de x
\end{align*}
vanishes, since this is an integral over an odd function converging to $0$ at $\pm \infty$. By Lemma~\ref{lem:nonsingular_equation}, $w_1$ is the unique solution of the inhomogeneous equation
\begin{align*}
	\mathfrak{L}_{\theta_0,\sigma_0} w_1 = 2 |u_0|^2 v_1 + u_0^2 \overline{v_1} - \iu \omega u_0' \in \ker(\mathfrak{L}_{\theta_0,\sigma_0})^\perp = \mathrm{ran}(\mathfrak{L}_{\theta_0,\sigma_0}).
\end{align*}
From now on, we denote by $\smash{\mathfrak{L}_{\theta_0,\sigma_0}^{-1}\colon \mathrm{ran}(\mathfrak{L}_{\theta_0,\sigma_0})}\to\ker(\mathfrak{L}_{\theta_0,\sigma_0})^\perp $ the inverse of the restricted self-adjoint operator $\smash{\mathfrak{L}_{\theta_0,\sigma_0}\colon \ker(\mathfrak{L}_{\theta_0,\sigma_0})^\perp \to \mathrm{ran}(\mathfrak{L}_{\theta_0,\sigma_0})}$. Clearly, $\smash{\mathfrak{L}_{\theta_0,\sigma_0}^{-1}}$ must also be self adjoint. This yields
\begin{align}\label{eq:I}
\begin{split}
    \Re \int_\R |u_0'|^2 u_0 \overline{w_1} \de x & = 
	\Re \int_\R |u_0'|^2 u_0 \overline{\mathfrak{L}_{\theta_0,\sigma_0}^{-1}(2 |u_0|^2 v_1 + u_0^2 \overline{v_1} - \iu \omega u_0')} \de x  \\
	& = \Re \int_\R \mathfrak{L}_{\theta_0,\sigma_0}^{-1}(|u_0'|^2 u_0) \overline{(2 |u_0|^2 v_1 + u_0^2 \overline{v_1} - \iu \omega u_0')} \de x.
\end{split}
\end{align}
Differentiating the equation
\begin{align*}
	0 = \mathfrak{L}_{\theta_0,\sigma_0} u_0' = - u_0''' + \zeta u_0' - 2 |u_0|^2 u_0' - u_0^2 \overline{u_0'},
\end{align*}
we obtain
\begin{align*}
	\mathfrak{L}_{\theta_0,\sigma_0} u_0'' = 6 |u_0'|^2 u_0 \quad \implies \quad
	\mathfrak{L}_{\theta_0,\sigma_0}^{-1}(|u_0'|^2 u_0 ) = \frac{1}{6} u_0''.
\end{align*}
Using $\iu u_0'' \overline{u_0} \in \iu\R$ and integrating by parts, this yields
\begin{align*}
	\eqref{eq:I}& = \frac{1}{6} \Re \int_\R u_0'' (2 |u_0|^2 \overline{v_1} + \overline{u_0}^2  v_1  + \iu \omega \overline{u_0'}) \de x \\
	& = \frac{1}{2} \Re \int_\R u_0'' |u_0|^2 \overline{v_1} \de x \\
	& = - \frac{1}{2} \Re \int_\R \big(2 |u_0'|^2 u_0 \overline{v_1} + u_0' |u_0|^2 \overline{v_1'} \big)\de x.
\end{align*}
In conclusion, we arrive at
\begin{align*}
	\langle \tilde{L}_1 \tilde{\mathbf{p}}_0, \tilde{\mathbf{p}}_0 \rangle_{L^2}
	= 3\Re\int_\R |u_0|^2 u_0' \overline{v_1'} \de x.
\end{align*}
Using $u_0 = \phi_{\theta_0,\sigma_0}$ and $v_1 = - \iu (-\partial_x^2+\zeta)^{-1}F$, we infer
\begin{align*}
    3\Re\int_\R |u_0|^2 u_0' \overline{v_1'} \de x 
    &=3 \Re \iu\eu^{\iu\theta_0} \int_\R (-\partial_x^2+\zeta)^{-1}(\phi_{0,\sigma_0}^2\phi_{0,\sigma_0}') \partial_x\overline{F} \de x \\
    &=\Re \iu\eu^{\iu\theta_0} \int_\R \phi_{0,0}' \overline{F'(x+\sigma_0)} \de x \\
    &= \Im \langle \eu^{-\iu\theta_0}F'(\cdot+\sigma_0),\phi_{0,0}'\rangle_{L^2}.
\end{align*}
Thus, we have established
$$
    \langle L_1 \mathbf{p}_0, \mathbf{p}_0 \rangle_{L^2}
    = \langle \tilde{L}_1 \tilde{\mathbf{p}}_0, \tilde{\mathbf{p}}_0 \rangle_{L^2}
	= \Im \langle \eu^{-\iu\theta_0}F'(\cdot+\sigma_0),\phi_{0,0}'\rangle_{L^2}
    =-\Im \langle \eu^{-\iu\theta_0}F''(\cdot+\sigma_0),\phi_{0,0}\rangle_{L^2}.
$$

Next, we compute $\langle \tilde{L}_1 \tilde{\mathbf{q}}_0, \tilde{\mathbf{p}}_0 \rangle_{L^2}$. First, we observe
\begin{align*}
	 \langle \tilde{L}_1 \tilde{\mathbf{q}}_0, \tilde{\mathbf{p}}_0 \rangle_{L^2}
	& = - \omega \int_\R |u_0'|^2 \de x -\iu \int_\R \big(u_0^2 \overline{u_1 u_0'} - \overline{u_0}^2 u_1 u_0' \big)\de x \\
	& = - \omega \int_\R |u_0'|^2 \de x - 2 \Re\int_\R \iu u_0^2 \overline{u_0'} \overline{(\iu \theta'(0) u_0 + \sigma'(0) u_0' + v_1+ w_1)} \de x \\
	& = - \omega \int_\R |u_0'|^2 \de x - 2 \Re \int_\R \big(\iu u_0^2 \overline{u_0' v_1} + \iu u_0^2 \overline{u_0' w_1}\big) \de x.
\end{align*}
It holds
\begin{align}\label{eq:III}
\begin{split}
	\Re \int_\R \iu u_0^2 \overline{u_0' w_1} \de x & =
	\Re \int_\R \iu u_0^2 \overline{u_0' \mathfrak{L}_{\theta_0,\sigma_0}^{-1}(2 |u_0|^2 v_1 + u_0^2 \overline{v_1} - \iu \omega u_0')} \de x \\
	& = \Re \int_\R \mathfrak{L}_{\theta_0,\sigma_0}^{-1}(\iu u_0^2 \overline{u_0'}) \overline{(2 |u_0|^2 v_1 + u_0^2 \overline{v_1} - \iu \omega u_0')} \de x.
\end{split}
\end{align}
Using that $u_0$ solves the NLS equation~\eqref{NLS}, we obtain
\begin{align*}
	\mathfrak{L}_{\theta_0,\sigma_0}(\iu u_0') =
	- \iu u_0'''+ \iu \zeta u_0' - 2\iu |u_0|^2 u_0' + \iu u_0^2 \overline{u_0'} 
	= \iu \mathfrak{L}_{\theta_0,\sigma_0}u_0' + 2 \iu u_0^2 \overline{u_0'} 
	= 2\iu u_0^2 \overline{u_0'},
\end{align*}
implying
\begin{align*}
	\mathfrak{L}_{\theta_0,\sigma_0}^{-1} (\iu u_0^2 \overline{u_0'}) =
	\frac{\iu}{2} u_0'.
\end{align*}
Thus, we deduce
\begin{align*}
	\eqref{eq:III}& = \frac{1}{2}\Re \int_\R \iu u_0' (2 |u_0|^2 \overline{v_1} + \overline{u_0}^2 v_1 + \iu \omega \overline{u_0'}) \de x \\
	& = \frac12 \Re\int_\R 2 \big(\iu u_0^2 \overline{u_0' v_1} + \iu u_0' \overline{u_0}^2 v_1\big)\de x - \frac{1}{2} \omega \int_\R |u_0'|^2 \de x \\
	& = \frac{1}{2} \Re\int_\R \iu u_0^2 \overline{u_0' v_1} \de x - \frac12 \omega \int_\R |u_0'|^2 \de x
\end{align*}
and arrive at
\begin{align*}
	\langle L_1 \mathbf{q}_0, \mathbf{p}_0 \rangle_{L^2} = \langle \tilde{L}_1 \tilde{\mathbf{q}}_0, \tilde{\mathbf{p}}_0 \rangle_{L^2} = - 3\Re \int_\R \iu u_0^2 \overline{u_0' v_1} \de x.
\end{align*}
Next, we compute
\begin{align*}
    \Re \int_\R \iu u_0^2 \overline{u_0' v_1} \de x 
    &=-\Re \eu^{\iu\theta_0} \int_\R (-\partial_x^2+\zeta)^{-1}(\phi_{0,\sigma_0}^2\phi_{0,\sigma_0}') \overline{F}\de x \\
    &=-\frac{1}{3}\Re \eu^{\iu\theta_0} \int_\R \phi_{0,\sigma_0}' \overline{F}\de x \\
    &= \frac{1}{3} \Re\langle \eu^{-\iu\theta_0} F'(\cdot+\sigma_0),\phi_{0,0}\rangle_{L^2},
\end{align*}
as well as
$$
    \int_\R |u_0'|^2 \de x = \frac{4}{3} \zeta^{3/2}.
$$
This yields the formula
$$
    \langle L_1 \mathbf{q}_0, \mathbf{p}_0 \rangle_{L^2} = -\Re\langle \eu^{-\iu\theta_0} F'(\cdot+\sigma_0),\phi_{0,0}\rangle_{L^2}.
$$
In addition, since it holds $L_1^* = L_1$, we infer 
\begin{align*}
	\langle L_1 \mathbf{q}_0, \mathbf{p}_0 \rangle_{L^2} = \langle \mathbf{q}_0, L_1 \mathbf{p}_0 \rangle_{L^2} = \overline{\langle L_1 \mathbf{p}_0, \mathbf{q}_0 \rangle_{L^2}} = \langle L_1 \mathbf{p}_0, \mathbf{q}_0 \rangle_{L^2}.
\end{align*}

We proceed with the computation of the last entry of the matrix $\mathcal{A}$. It holds
\begin{align*}
	\langle \tilde{L}_1 \tilde{\mathbf{q}}_0, \tilde{\mathbf{q}}_0 \rangle_{L^2}
	& = - \int_\R \big( |u_0|^2 u_0 \overline{u_1} + |u_0|^2 \overline{u_0} u_1 \big)\de x \\
	& = -2 \Re \int_\R |u_0|^2 u_0 \overline{(\iu \theta'(0) u_0 + \sigma'(0) u_0' + v_1+ w_1)} \de x \\
	& = -2 \Re\int_\R \big( |u_0|^2 u_0 \overline{v_1} + |u_0|^2 u_0 \overline{w_1}\big) \de x.
\end{align*}
Again, we use the inhomogeneous equation for $w_1$ to compute
\begin{align}\label{eq:II}
	\Re\int_\R |u_0|^2 u_0 \overline{w_1} \de x & = 
	\Re \int_\R |u_0|^2 u_0 \overline{\mathfrak{L}_{\theta_0,\sigma_0}^{-1}(2 |u_0|^2 v_1 + u_0^2 \overline{v_1} - \iu\omega u_0')} \de x \notag\\
	& = \Re \int_\R \mathfrak{L}_{\theta_0,\sigma_0}^{-1}(|u_0|^2 u_0) \overline{(2 |u_0|^2 v_1 + u_0^2 \overline{v_1} - \iu\omega u_0')} \de x.
\end{align}
Using that $u_0$ solves the NLS equation~\eqref{NLS}, we have
\begin{align*}
	\mathfrak{L}_{\theta_0,\sigma_0} u_0 = - u_0'' + \zeta u_0 - 2 |u_0|^2 u_0 - u_0^2 \overline{u_0} = - 2 |u_0|^2 u_0,
\end{align*}
which implies
\begin{align*}
	\mathfrak{L}_{\theta_0,\sigma_0}^{-1}(|u_0|^2 u_0) = -\frac{1}{2} u_0.
\end{align*}
In particular, we establish
\begin{align*}
	\eqref{eq:II}& = - \frac{1}{2} \Re\int_\R u_0 (2 |u_0|^2 \overline{v_1} + \overline{u_0}^2  v_1 + \iu \omega \overline{u_0'}) \de x\\
	& = - \frac{1}{2} \Re\int_\R \big( 2 |u_0|^2 u_0 \overline{v_1} + |u_0|^2 \overline{u_0} v_1 \big) \de x \\
	& = - \frac{3}{2} \Re \int_\R |u_0|^2 u_0 \overline{v_1} \de x
\end{align*}
and, thus, we arrive at the formula
\begin{align*}
	\langle \tilde{L}_1 \tilde{\mathbf{q}}_0, \tilde{\mathbf{q}}_0 \rangle_{L^2} = \Re \int_\R |u_0|^2 u_0 \overline{v_1} \de x.
\end{align*}
Moreover, we establish
\begin{align*}
     \Re \int_\R |u_0|^2 u_0 \overline{v_1} \de x
     & =\Re \iu \eu^{\iu \theta_0} \int_\R (-\partial_x^2+\zeta)^{-1}(\phi_{0,\sigma_0}^3)  \overline{F} \de x \\
     &=\Re \iu \eu^{\iu \theta_0} \int_\R \phi_{0,\sigma_0}  \overline{F} \de x \\
     &= \Im \langle \eu^{-\iu\theta_0} F(\cdot+\sigma_0),\phi_{0,0}\rangle_{L^2}
\end{align*}
and, hence, obtain 
$$
    \langle L_1 \mathbf{q}_0, \mathbf{q}_0 \rangle_{L^2} = 
    \langle \tilde{L}_1 \tilde{\mathbf{q}}_0, \tilde{\mathbf{q}}_0 \rangle_{L^2} =
    \Im\langle \eu^{-\iu\theta_0} F(\cdot+\sigma_0),\phi_{0,0}\rangle_{L^2}.
$$

We proceed with computing the entries of the matrix $\mathcal{B}$, given by~\eqref{eq:matrices_reduces_evp}. Using integration by parts, we deduce
\begin{align*}
    \langle J\mathbf{p}_1,\mathbf{p}_0\rangle_{L^2} &=
    -\int_\R \frac{x}{2} \phi_{0,\sigma_0} \phi_{0,\sigma_0}' \de x = \frac14 \|\phi_{0,0}\|_{L^2}^2 = \sqrt{\zeta}.
\end{align*}
Moreover, we infer
\begin{align*}
    \langle J\mathbf{p}_1,\mathbf{q}_0\rangle_{L^2} &=
     -\int_\R \frac{x}{2} \phi_{0,\sigma_0}^2 
    \left\langle J 
    \begin{pmatrix}
        \cos(\theta_0) \\ \sin(\theta_0)
    \end{pmatrix},
    \begin{pmatrix}
        \cos(\theta_0) \\ \sin(\theta_0)
    \end{pmatrix}
    \right\rangle_{\C^2} \de x
    = 0,
\end{align*}
and, similarly, $\langle J\mathbf{q}_1,\mathbf{p}_0\rangle_{L^2} = 0$. Finally, we compute
\begin{align*}
    \langle J\mathbf{q}_1,\mathbf{q}_0\rangle_{L^2} &= \langle J \partial_\zeta \boldsymbol{\phi}_{\theta_0,\sigma_0},-J\boldsymbol{\phi}_{\theta_0,\sigma_0} \rangle_{L^2}
    =- \langle \partial_\zeta \phi_{0,\sigma_0},\phi_{0,\sigma_0} \rangle_{L^2}
    = -\frac{1}{2} \partial_\zeta \|\phi_{0,0}\|_{L^2}^2 = - \frac{1}{\sqrt{\zeta}},
\end{align*}
which completes the proof.
\end{proof}

\bibliographystyle{abbrv}
\bibliography{bibliography}

\end{document}